\documentclass[11pt]{article}
\usepackage{moreverb}
\usepackage{amsmath}
\usepackage{amsfonts}
\usepackage{amsthm}
\usepackage{graphicx}
\usepackage{fullpage}

\newtheorem{Theorem}{Theorem}[section]

\newtheorem{Proposition}[Theorem]{Proposition}
\newtheorem{Lemma}[Theorem]{Lemma}
\newtheorem{Cor}[Theorem]{Corollary}

\numberwithin{equation}{section}

\newcommand{\ds}{\displaystyle}

\newcommand{\bs}[1]{\boldsymbol{#1}}

\newcommand{\cP}{\mathcal{P}}
\newcommand{\cQ}{\mathcal{Q}}
\newcommand{\cS}{\mathcal{S}}
\newcommand{\cX}{\mathcal{G}}
\title{Connection coefficients and 
monodromy representations 
for a class of Okubo systems of ordinary 
differential equations}
\author{
Shotaro KONNAI \\
{\normalsize Department of Mathematics, Kobe University}
}
\date{}
\begin{document}
\maketitle

\begin{abstract}
Explicit connection coefficients and monodromy 
representations 
are constructed for the canonical solution matrices 
of a class of Okubo systems of 
ordinary differential equations 
as an application of the Katz operations.  
\end{abstract}

\arraycolsep=2pt
\maketitle

\section*{Introduction}
The Okubo systems are one of the most important classes 
of Fuchsian systems of ordinary differential equations 
on the Riemann sphere $\mathbb{P}^1$; 
for the definition of an Okubo system, see Section 1.
A characteristic feature of this class is that 
each Okubo system admits a solution matrix consisting 
of non-holomorphic 
solutions around its finite singular points, 
which we call below {\em Okubo's canonical solution matrix}. 
Okubo \cite{Okubo} studied basic properties 
and in particular established an explicit determinant 
formula for such a solution matrix (see Theorem \ref{det:Okubo1}).  
This determinant formula guarantees that Okubo's 
canonical solution matrix is in fact a fundamental 
solution matrix.  
Also, it is because of this determinant formula that 
Okubo systems provide with good models for global 
analysis of Fuchsian systems. 

Yokoyama \cite{TY} classified the types of 
irreducible rigid Okubo systems 
under the condition that the nontrivial local exponents 
are mutually distinct at each finite singular point; 
this class consists of 
eight types 
${\rm I}$, ${\rm II}$, 
${\rm III}$, ${\rm IV}$ and 
${\rm I}^\ast$, ${\rm II}^\ast$, 
${\rm III}^\ast$, ${\rm IV}^\ast$, 
which is usually referred to as {\em Yokoyama's list}.  
For each type in Yokoyama's list, 
Haraoka constructed 
a canonical form of the Okubo system
\cite{YHEq}, as well as the corresponding 
monodromy representation, 
up to the action of diagonal matrices \cite{YHMon}.  
The purpose of this paper is to propose a method 
for determining the 
explicit monodromy representation for Okubo's 
canonical solution matrix, 
including the diagonal matrix factor
which has not been fixed in \cite{YHMon}.  
This problem is in fact reduced to determining 
the connection coefficients 
among non-holomorphic solutions for the Okubo system. 
We solve this connection problem by means of the middle convolutions 
for Schlesinger systems.  
In this paper we deal with the particular cases of types ${\rm I}$, ${\rm I}^*$, ${\rm II}$
and ${\rm III}$; other cases will be discussed in our forthcoming paper. 

\par\medskip
It is known by a celebrated work of 
Katz \cite{Katz} that any irreducible rigid 
local system can be constructed by a finite iteration of the 
so-called Katz operations 
(additions and middle convolutions)
from a local system of rank one.  
Dettweiler and Reiter \cite{DR2007}
reformulated 
the Katz operations 
so that they can be applied 
directly to Schlesinger systems and their monodromy 
representations.  
Therefore, any irreducible 
rigid Schlesinger system, as well as its 
monodromy representation, can be constructed 
in principle 
from the rank one case. 
In the case of Okubo systems, Yokoyama 
independently introduced the notions of 
extending and restricting operations for differential systems
and their monodromy matrices \cite{TY2}.  
He also proved that any generic rigid Okubo system
can be constructed by a finite iteration of his operations. 
It is clarified by Oshima \cite{OY}  
the Katz operations and the Yokoyama operations are essentially 
equivalent as far as the Okubo systems are concerned. 
It is not completely clarified, however, how the Katz (or Yokoyama)  
operations can be realized on the level of fundamental solution matrices 
and their monodromy representations. 

Our method for solving the connection problem is based on the 
inductive construction of Okubo systems by middle convolutions.  
For a given 
Schlesinger system and a fundamental solution matrix, 
we give an explicit construction of a fundamental solution matrix 
of its middle convolution in the sense of \cite{DR2007}. 
We then apply this procedure for constructing the Okubo systems 
in Yokoyama's list and their explicit monodromy representations. 

\par\medskip
This paper is organized as follows.  
We propose in Section 1 the notion of {\em Okubo's canonical solution matrix} 
$\Psi(x)$ for 
an Okubo system
\begin{equation}\label{eq:InOkubo}
(x-T)\frac{d}{dx}Y=AY, \quad A\in{\rm Mat}(n;\mathbb{C})
\end{equation}
of ordinary differential equations. 
This solution matrix, introduced by Okubo \cite{Okubo}, 
consists of non-holomorphic solutions around finite singular points, and 
forms a fundamental solution matrix under a certain genericity condition. 
We also give a remark on the relation  between 
the connection coefficients among non-holomorphic local solutions 
and the monodromy matrices for Okubo's canonical solution matrix. 
The monodromy matrices for $\Psi(x)$ 
are determined from the connection coefficients through formula 
\eqref{eq:Mk}. 
Our main results are collected in Section 2.  
We fix a canonical form for each of the Okubo systems of 
types ${\rm I}$, ${\rm I}^*$, ${\rm II}$ and ${\rm III}$
of Yokoyama's list as in Theorems 2.1, 2.2 and 2.3, and 
give the explicit monodromy matrices 
for the corresponding canonical solution matrix 
in Theorem 2.4, 2.5 and 2.6.
Theses results are proved in subsequent sections 
by iterations of the Katz operations.  

 Our method for determining the connection coefficients and the 
 monodromy matrices is based on the middle convolution for Schlesinger 
 systems of \cite{DR2007}.  
The procedure of the middle convolution for a Schlesinger system
\begin{equation}\label{eq:InSchl}
\frac{d}{dx}y=\sum_{k=1}^r\frac{A_k}{x-t_k}Y
\end{equation} 
is divided into three steps, which we call the 
{\em convolution}, the {\em K-reduction} and the {\em L-reduction}.  
Analyzing these three steps, in Section 3 we clarify how one can construct a 
fundamental solution matrix for the middle convolution from a 
given fundamental solution matrix of the system \eqref{eq:InSchl}.  

A special combination of Katz operations 
\begin{equation}
{\rm add}_{(0,\ldots,\rho,\ldots,0)}\circ{\rm mc}_{-\rho-c}\circ{\rm add}_{(0,\ldots,c,\ldots,0)}
\end{equation}
(middle convolution with additions at a singular points)  
is used effectively for constructing the Okubo systems in 
Yokoyama's list.  
In fact, 
each Okubo system in Yokoyama's list can be constructed 
by a finite iteration of such operations from the rank one case. 
We formulate in Section 4 this type of operations 
for a general Okubo system and its monodromy representation.
We show that the resulting Schlesinger system is also an Okubo system 
as explicitly given by \eqref{eq:mcadd}.
For the monodromy representation,  \eqref{Mon:MCadd} provides the moodromy representation
 of the resulting system specified by \eqref{eq:mcadd}.  
Furthermore we obtain the connection coefficients for the canonical solution matrix of the resulting system as in Theorem \ref{connection} by applying these results.  

In Section 5 we prove our main theorems for Okubo systems 
of type ${\rm I}$, ${\rm I^\ast}$, ${\rm II}$ and ${\rm III}$.
Our proof is divided into two steps.
We first compute a certain part of the connection coefficients for 
each of those Okubo systems by applying Theorem \ref{connection} 
repeatedly. 
Second, we determine the other connection coefficients by 
the symmetry of the Okubo system.
Our argument is based on the fact 
that the permutaion of characteristic exponents at a singular 
point can be 
realized by the adjoint action of a contant matrix. 

\par\medskip

We remark that our approach is similar to that of Yokoyama \cite{TY3} 
in which recursive relations for connection coefficients are investigated  
in the framework of extending operations for Okubo systems. 
It is not clear, however, whether the connection coefficients for individual 
Okubo systems in Yokoyama's list can be directly determined only 
from the results of \cite{TY3}.  
We also remark that
the connection problem 
for rigid irreducible Fuchsian differential equations of scalar type 
has been discussed by 
Oshima \cite{TO} as an application of the Katz operations. 

\tableofcontents

\section{Okubo systems and their canonical solution matrices}

In this section we formulate the notion of a 
{\em canonical solution matrix} 
for a system of ordinary differential equations 
of Okubo type, and investigate 
fundamental properties of its connection coefficients and  monodromy matrices.  

\subsection{Canonical solution matrix for an Okubo system}
Let $t_1,\ldots,t_r$ be $r$ distinct points in $\mathbb{C}$.  
Fixing an $r$-tuple $(n_1,\ldots,n_r)$ of positive integers with 
$n_1+\cdots+n_r=n$, we denote by 
\begin{equation}
T=\mbox{\rm diag}(t_1I_{n_1},t_2I_{n_2},\ldots,t_rI_{n_r})=
\begin{pmatrix}
t_1I_{n_1} & \\
& t_2I_{n_2} &\\
&&\ddots\\
&&& t_rI_{n_r}
\end{pmatrix}
\end{equation}
the block diagonal matrix whose diagonal blocks are scalar matrices 
$t_i I_{n_i}$ ($i=1,\ldots,r$), where $I_k$ 
stands for the $k\times k$ identity 
matrix.  
A system of ordinary differential equations on $\mathbb{P}^1$ of the form 
\begin{equation}\label{eq:Okubo}
(xI_n-T)\frac{d}{dx}Y=A\,Y\qquad 
(A\in\mbox{Mat}(n;\mathbb{C})), 
\qquad 
\end{equation}
is called an {\em Okubo system} of type $(n_1,\ldots,n_r)$, 
where $Y=(y_1,\ldots,y_m)^{\scriptsize\mbox{t}}$ is the 
column vector of unknown functions.  
For each $i,j=1,\ldots r$, we denote by $A_{ij}$
the $(i,j)$-block of the matrix $A$: 
\begin{equation}\label{eq:OkA}
A=\begin{pmatrix}
A_{11} & \ldots & A_{1r} \\
\vdots & & \vdots \\ 
A_{r1} & \ldots & A_{rr}
\end{pmatrix},\qquad 
A_{ij}\in\mbox{Mat}(n_i,n_j;\mathbb{C}).  
\end{equation}
Then the Okubo system \eqref{eq:Okubo} can be equivalently 
rewritten in the Schlesinger form 
\begin{equation}\label{eq:Sch}
\frac{d}{dx} Y=\left(\sum_{k=1}^{r}\frac{A_k}{x-t_k}\right)Y, 
\qquad
A_k
=
\begin{pmatrix}
&{\large O} & \\ 
A_{k1} & \ldots & A_{kr}\\[2pt]
&{\large O} & 
\end{pmatrix}\quad(k=1,\ldots,r). 
\end{equation}

We denoting by $\alpha^{(k)}_{1},\ldots,\alpha^{(k)}_{n_k}$ 
the eigenvalues of the diagonal block $A_{kk}$ ($k=1,\ldots,r$) 
and by $\rho_1,\ldots,\rho_n$ the eigenvalues of $A=A_1+\cdots+A_r$; 
these eigenvalues are subject to the {\em Fuchs relation}
\begin{equation}
\sum_{k=1}^{r}\sum_{j=1}^{n_k}\alpha^{(k)}_j=\sum_{j=1}^{n}\rho_j.
\end{equation}
We assume hereafter that 
\begin{equation}\label{assumedet1}
\alpha^{(k)}_{i}-\alpha^{(k)}_{j}\notin 
\mathbb{Z}\backslash \{0\}
\quad(1\le i<j\le n_k),\quad
\alpha^{(k)}_{j}\notin \mathbb{Z}\quad(1\le j\le n_k)
\end{equation}
for $k=1,\ldots,r$. 

\par\medskip
A characteristic feature of a system of Okubo type 
is that, if the local exponents are generic, there exists 
a {\em canonical} fundamental solution matrix 
consisting of singular solutions around 
the finite singular points $x=t_1,\ldots,t_r$. 

In what follows, we 
denote by $\mathcal{O}(\widetilde{\mathcal{D}})$ the vector space 
of all multivalued holomorphic functions on 
$\mathcal{D}=\mathbb{C}\backslash\{t_1,\ldots,t_r\}$ and 
by $\mathcal{O}_a$ the ring of germs of homomorphic functions 
at $x=a$. 
Fixing a base point $p_0\in D$, 
we identify a multivalued holomorphic function on $D$ with 
a germ of holomorphic function at $x=p_0$ which can be continued 
analytically along any continuous path in $D$ starting from $p_0$. 
We denote by 
$\gamma_k$ ($k=1,\ldots,r$) and $\gamma_\infty$ the homotopy classes 
in $\pi_1(\mathcal{D},p_0)$ of continuous paths 
encircling $x=t_k$ and $x=\infty$ 
in the positive direction, respectively; 
we choose these paths so that $\gamma_\infty \gamma_1\cdots \gamma_r=1$ 
in $\pi_{1}(\mathcal{D},p_0)$ (See Figure 1).  
In this paper, for two continuous paths $\alpha,\beta$,  we use the notation 
$\beta\alpha$ to refer to the path obtained by connecting 
$\alpha$ and $\beta$ in this order. 
Choosing the base point $p_0$ appropriately, we assume 
that ${\rm Im}\,(t_i-p_0)/(t_1-p_0)<0$ for $i=2,\ldots,r$, 
and assign the arguments 
$\theta_k={\rm arg}(p_0-t_k)$ 
($k=1,2,\ldots,r$) so that 
\begin{equation}
\theta_1>\theta_2>\cdots>\theta_r>\theta_1-\pi.  
\end{equation}
When we consider the behavior of 
$f\in\mathcal{O}(\widetilde{\mathcal{D}})$ around  $x=t_k$, 
we use the analytic continuation of $f\in\mathcal{O}_{p_0}$ 
along the line segment connecting $p_0$ and $t_k$. 
\begin{figure}[htbp]
\centering
\unitlength=1.2pt
\begin{picture}(80,80)(0,-14)
\put(30,46){\circle*{3}}
\put(80,46){\circle*{3}}
\put(178,46){\circle*{3}}
\put(120,43){$\cdots$}
\put(20,48){$t_1$}
\put(70,48){$t_2$}
\put(180,48){$t_r$}
\put(122,-10){$p_0$}
\put(35,25){$\gamma_1$}
\put(110,30){$\gamma_2$}
\put(165,18){$\gamma_r$}
\put(70,34){\vector(-2,1){2}}
\put(103,32){\vector(-2,3){2}}
\put(170,29){\vector(3,2){2}}
\end{picture}
\includegraphics[scale=0.85, bb=115 -10 300 100]{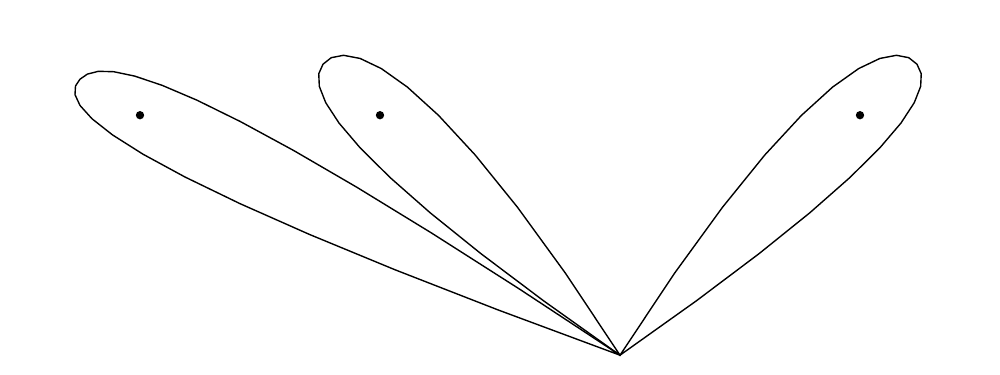}
\caption{}
\end{figure}

If condition \eqref{assumedet1} is satisfied, 
for each $k=1,\ldots,r$, 
system \eqref{eq:Okubo} has a unique fundamental 
solution matrix $\Psi^{(k)}(x)\in {\rm Mat}(n;\mathcal{O}(\widetilde{\mathcal{D}}))$ of the form 
\begin{equation}
\Psi^{(k)}(x)=F^{(k)}(x)(x-t_k)^{A_{k}},\quad
 F^{(k)}(x)\in {\rm Mat}(n;\mathcal{O}_{t_k}),
\quad  F^{(k)}(t_k)=I_n, 
\end{equation}
where we specify the branch of 
$(x-t_k)^{A_k}=\exp(A_k\log(x-t_k))$ near $x=p_0$ by 
${\rm arg}(p_0-t_k)=\theta_k$.  
We call $\Phi^{(k)}(x)$ the {\em local canonical solution matrix} of \eqref{eq:Okubo} at $x=t_k$.  
We decompose this $\Psi^{(k)}(x)$ as
\begin{equation}
\Psi^{(k)}(x)=(\Psi_1^{(k)}(x),\ldots,\Psi_r^{(k)}(x)),\quad 
\Psi_j^{(k)}(x) \in {\rm Mat}(n,n_j;\mathcal{O}(\widetilde{\mathcal{\mathcal{D}}})).
\end{equation}
Then the $k$th block $\Psi_k^{(k)}(x)$ represents a basis 
of all singular solutions around $x=t_k$, and its local monodromy 
is given by 
\begin{equation}
\gamma_k\!\cdot\!\Psi^{(k)}_k(x)=\Psi^{(k)}_k(x)e(A_{kk}), 
\end{equation}
where $e(\mu)=\exp(2\pi i\mu)$. 
Collecting the singular solutions $\Psi^{(k)}_{k}(x)$ ($k=1,\ldots,r$) 
together, 
we define the $n\times n$ solution matrix $\Psi(x)$ by
\begin{equation}
\Psi(x)=(\Psi^{(1)}_1(x),\Psi^{(2)}_2(x),\ldots,\Psi^{(r)}_r(x))
\in{\rm Mat}(n;\mathcal{O}(\widetilde{\mathcal{D}})). 
\end{equation}
We call this $\Psi(x)$ the {\em canonical solution matrix} 
for the Okubo system \eqref{eq:Okubo}.  

As for the determinant of $\Psi(x)$, the following theorem 
is known as Okubo's determinant formula 
(Okubo \cite{Okubo}, Kohno \cite{Kohno}).  
\begin{Theorem}\label{det:Okubo1}
Under the assumption \eqref{assumedet1}, 
the determinant of the canonical solution matrix $\Psi(x)$ 
is explicitly given by 
\begin{equation}
{\rm det}(\Psi(x))=
\frac{\prod_{k=1}^r\prod_{j=1}^{n_k}\Gamma(1+\alpha^{(k)}_{j})}
{\prod_{i=1}^n\Gamma(1+\rho_i)}
\prod_{k=1}^{r}(x-t_k)^{\sum_{j=1}^{n_k}\alpha^{(k)}_j}. 
\end{equation}
 \end{Theorem}

\begin{Cor}\label{det:Okubo2}
Under the assumption \eqref{assumedet1}, 
the canonical solution matrix $\Psi(x)$ 
of the Okubo system \eqref{eq:Okubo} is a 
fundamental solution matrix
if and only if $\rho_i\notin \mathbb{Z}_{\le 0}$ 
for $i=1,\ldots,n$. 
\end{Cor}

\subsection{Connection coefficients and monodromy matrices}

We  give some remarks 
on the relationship between the connection coefficients and 
the monodromy
matrices for the canonical solution matrix $\Psi(x)$ 
of the Okubo system \eqref{eq:Okubo}.

Let $\Psi(x)$ be the canonical solution matrix of the Okubo system \eqref{eq:Okubo}. 
The analytic continuation of $\Psi(x)$ by $\gamma_k$ 
($k=1,\ldots,r$) 
defines an $r$-tuple of monodromy matrices 
${\bs M}=(M_1,\ldots,M_r)
\in{\rm GL}(n;\mathbb{C})^r$ as 
\begin{equation}
\gamma_k\!\cdot\!\Psi(x)=\Psi(x)M_k \quad(k=1,\ldots,r). 
\end{equation}
Through the analytic continuation along $\gamma_k$ 
to a neighborhood of $x=t_k$, 
the $j$-th column block $\Psi^{(j)}_j(x)$ ($j=1,\ldots,r$)
is expressed uniquely in the form 
\begin{equation}\label{eq:Ckj}
\Psi_{j}^{(j)}(x)=\Psi_k^{(k)}(x)C^{(kj)}+H_j^{(k)}(x), 
\end{equation}
where $C^{(kj)}$ is an $n_k\times n_j$ constant matrix and 
$H_j^{(k)}(x)$ is an $n\times n_j$ solution matrix which is 
holomorphic around $x=t_k$. 
Since $\gamma_k\!\cdot\!\Psi^{(k)}_k(x)=\Psi^{(k)}_k(x)e(A_{kk})$, 
we have 
\begin{equation}
\gamma_k\!\cdot\!\Psi_j^{(j)}(x)=\Psi^{(k)}_k(x) e(A_{kk})C^{(kj)}+H^{(k)}_j(x), 
\end{equation}
and hence
\begin{equation}
\gamma_k\!\cdot\!\Psi_j^{(j)}(x)=\Psi_k^{(k)}(x)(e(A_{kk})-1)C^{(kj)}+\Psi_{j}^{(j)}(x)
\end{equation}
eliminating $H^{(k)}_j(x)$ by \eqref{eq:Ckj}. 
To summarize, the monodromy matrices for the 
canonical solution matrix $\Psi(x)$ are given as 
\begin{equation}\label{eq:Mk}
M_k=
\begin{pmatrix}
1 & & &  &\\[-4pt]
&\ddots & &  &\\
(e(A_{kk})-1)C^{(k1)}&\dots &e(A_{kk}) & \ldots
 &(e(A_{kk})-1)C^{(kr)} \\
 & & &\ddots&\\[-4pt]
 & & &  &1
\end{pmatrix}\quad(k=1,\ldots,r)
\end{equation}
in terms of the connection coefficients as in \eqref{eq:Ckj}. 
We remark that, in order to compute 
the monodromy matrices of $\Psi(x)$, we need not explicitly 
determine the holomorphic 
part $H^{(k)}_j(x)$ of analytic continuation.  
 
\section{Okubo systems of types ${\rm I}$, ${\rm I}^*$, ${\rm II}$ 
and ${\rm III}$}
\subsection{Canonical forms of the systems  of types  
${\rm I}$, ${\rm I}^*$, ${\rm II}$ and ${\rm III}$}
Yokoyama's list is a class of rigid irreducible Okubo systems 
such that 
each diagonal block
$A_{ii}\in{\rm Mat}(n_i;\mathbb{C})$ $(i=1,\ldots,r)$ of $A$ 
has $n_i$ mutually distinct eigenvalues 
and that the matrix 
$A\in{\rm Mat}(n;\mathbb{C})$ is diagonalizable. 
In this class we can assume that 
the diagonal blocks $A_{ii}$ ($i=1,\ldots,r$) 
are diagonal matrices with mutually distinct entries. 
We assume that $A$ is diagonalized as
\begin{equation}
A\sim{\rm diag}(\rho_1I_{m_1},\ldots,\rho_{q}I_{m_q})\quad (m_1\geq\cdots \geq m_q)
\end{equation} 
with mutually distinct $\rho_1,\ldots,\rho_q\in\mathbb{C}$. 
Then the 
Okubo systems in Yokoyama's list 
are characterized  
by the following eight pairs $(n_1,\ldots,n_r)$, 
$(m_1,\ldots,m_q)$ of partitions of $n$ respectively: 
\begin{equation}
\begin{array}{clcl}
({\rm I})_n:
&
\begin{cases}
(n_1,n_2)=(n-1,1)\\
(m_1,\ldots,m_n)=(1,\ldots,1)
\end{cases}\quad 
&
({\rm I}^*)_n:
&
\begin{cases}
(n_1,\ldots,n_n)=(1,\ldots,1)\\
(m_1,m_2)=(n-1,1)
\end{cases}\quad 
\\[16pt]
({\rm II})_{2n}:
&
\begin{cases}
(n_1,n_2)=(n,n)\\
(m_1,m_2,m_3)=(n,n-1,1)
\end{cases}\quad 
&
({\rm II^\ast})_{2n}:
&
\begin{cases}
(n_1,n_2,n_3)=(n,n-1,1)\\
(m_1,m_2)=(n,n)
\end{cases}\quad 
\\[16pt]
({\rm III})_{2n+1}:
&
\begin{cases}
(n_1,n_2)=(n+1,n)\\
(m_1,m_2,m_3)=(n,n,1)
\end{cases}
&
({\rm III^\ast})_{2n+1}:
&
\begin{cases}
(n_1,n_2,n_3)=(n,n,1)\\
(m_1,m_2)=(n+1,n)
\end{cases}
\\[16pt]
({\rm IV})_{6}:
&
\begin{cases}
(n_1,n_2)=(4,2)\\
(m_1,m_2,m_3)=(2,2,2)
\end{cases}\quad 
&
({\rm IV^\ast})_{6}:
&
\begin{cases}
(n_1,n_2,n_3)=(2,2,2)\\
(m_1,m_2)=(4,2)
\end{cases} 
\end{array}
\end{equation}
In this paper
we deal with 
the four types 
$({\rm  I})_{n}$, 
$({\rm  I^\ast})_{n}$, 
$({\rm  II})_{2n}$, 
$({\rm  IIII})_{2n+1}$
among these eight types, 

Canonical forms of the Okubo systems in Yokoyama's list are determined by 
the work of Haraoka \cite{YHEq}.   
We first specify explicit canonical forms of the Okubo systems 
of types $({\rm I})_n$, $({\rm I}^*)_n$, $({\rm II})_{2n}$,\, $({\rm III})_{2n+1}$ 
$(n=2,3,\ldots)$ 
as in \cite{YHEq}, 
which we will use throughout this paper. 

\par\medskip
\noindent{\bf Case ${\rm I}$}:
We express the Okubo system of type $({\rm I})_n$ in the form
\begin{equation}\label{eq:OkuboA}
(x-T)\frac{d}{dx}Y=AY,\quad 
A=
\begin{pmatrix}
\alpha & K \\
 L     & \beta
\end{pmatrix}
=
\begin{pmatrix}
\alpha_1 &&& K_1 \\
& \ddots & & \vdots\\
&&\alpha_{n-1} & K_{n-1}\\
 L_1& \ldots & L_{n-1}  & \alpha_n
\end{pmatrix}
\sim
{\rm diag}(\rho_1,\rho_2,\ldots,\rho_n)\, , 
\end{equation}
where $\alpha,\,\beta $ and $T$ are diagonal matrices 
defined by
\begin{equation}
T={\rm diag}(t_1I_{n-1},t_2),\quad 
\alpha={\rm diag}(\alpha_1,\cdots,\alpha_{n-1}),\ \  \beta=(\alpha_n). 
\end{equation}
The Fuchs relation in this case is  given by 
\begin{equation}
\sum_{i=1}^{n}\alpha_i=\sum_{i=1}^{n}\rho_i.  
\end{equation}
We assume that the parameters of $({\rm I})_n$
satisfy the following conditions: 
\begin{equation}
\begin{array}{ll}
\alpha_i-\alpha_j\notin\mathbb{Z}\quad(1\le i<j\le n-1),
\quad & \alpha_i\notin\mathbb{Z}\quad(1\le i\le n),
\\[4pt]
\rho_i-\rho_j\notin\mathbb{Z}\quad(1\le i<j\le n),
\quad&\rho_i\notin\mathbb{Z}\quad(1\le i\le n). 
\end{array}
\end{equation}
For the type $({\rm I})_n$, the canonical form can be taken  as follows
\begin{Theorem}[\cite{YHEq}]\label{eq:cI}
Canonical form of the 
Okubo system of type $({\rm I})_{n}$ is 
given by the matrices $K=(K_{i})_{i}$ and $L=(L_{j})_{j}$ with the following entries respectively $:$ 
\begin{equation}\label{eq:CI}
K_i=1,\qquad L_j=-\frac{\prod_{k=1}^n(\alpha_j-\rho_k)}
{\prod_{k\neq j\,1\leq k\leq n-1}(\alpha_j-\alpha_k)},
\quad (i,j=1,\ldots,n-1)
\end{equation}
\end{Theorem}

\par\medskip
\noindent{\bf Case ${\rm I}^*$}:
The Okubo system of type $({\rm I}^*)_n$ is expressed in the form
\begin{equation}
\begin{pmatrix}
x-t_1&&\\
&\ddots&\\
&&x-t_n
\end{pmatrix}
\frac{d}{dx}Y=\begin{pmatrix}
\alpha_1&a_{12}&\cdots &a_{1n}\\
a_{21}&\alpha_2&\cdots&a_{2n}\\
\vdots&&\ddots&\vdots\\
a_{n1}&a_{n2}&\cdots&\alpha_n
\end{pmatrix}Y
\sim{\rm diag}(\rho_1I_{n-1},\rho_2).
\end{equation}
Then the Fuchs relation is given by
\begin{equation}
\sum_{k=1}^n\alpha_k=(n-1)\rho_1+\rho_2.
\end{equation}
We assume that the parameters of $({\rm I}^*)_n$ satisfy the following conditions:
\begin{equation}
\alpha_i\notin \mathbb{Z}\quad (1\leq i\leq n),\quad \rho_1-\rho_2\notin  \mathbb{Z},
\quad \rho_1,\,\rho_2\notin  \mathbb{Z},
\end{equation}
The canonical forms of the type $({\rm I}^*)_n$ is given by following form:
\begin{Theorem}[\cite{YHEq}]
Canonical form of the 
Okubo system of type $({\rm I}^*)_{n}$ is 
given by the components $a_{ij}$ with the following entries respectively $:$ 
\begin{equation}\label{eq:CI*}
a_{ij}=\alpha_j-\rho_1
\quad (i\neq j,\,\,i,j=1,\ldots,n)
\end{equation}
\end{Theorem}

\par\medskip
\noindent{\bf Cases ${\rm II}$ and ${\rm III}$}:
We express the Okubo systems of type $({\rm II})_{2n}$ and 
$({\rm III})_{2n+1}$ in the form 
\begin{equation}\label{eq:OkuboA}
(x-T)\frac{d}{dx}Y=AY,\quad 
A=
\begin{pmatrix}
\alpha & K \\
 L     & \beta
\end{pmatrix}\sim
{\rm diag}(\rho_1I_n,\rho_2I_{m-1},\rho_3)\ \  (m=n, n+1),
\end{equation}
where $\alpha,\,\beta $ and $T$ are diagonal matrices 
defined by
\begin{equation}
\begin{split}
({\rm II})_{2n}:\ \ 
&T={\rm diag}(t_1I_n,t_2I_n),\quad 
\alpha={\rm diag}(\alpha_1,\cdots,\alpha_n),\quad  
\beta={\rm diag}(\beta_1,\cdots,\beta_n),\\
({\rm III})_{2n+1} :\ \ 
&T={\rm diag}(t_1I_{n+1},t_2I_{n}),\quad 
\alpha={\rm diag}(\alpha_1,\cdots,\alpha_{n+1}),\quad  
\beta={\rm diag}(\beta_1,\cdots,\beta_{n}),
\end{split}
\end{equation}
respectively. 
Note that the Fuchs relations in these cases are  given by 
\begin{equation}
\sum_{i=1}^{m}\alpha_i+\sum_{i=1}^{n}\beta_i=n\rho_1+(m-1)\rho_2+\rho_3.  
\end{equation}
We always assume that the parameters 
satisfy the following conditions: 
\begin{equation}
\begin{array}{ll}
\alpha_i-\alpha_j\notin\mathbb{Z}\quad(1\le i<j\le m),
\quad & \alpha_i\notin\mathbb{Z}\quad(1\le i\le m),
\\[4pt]
\beta_i-\beta_j\notin\mathbb{Z}\quad(1\le i<j\le n),
\quad &\beta_i\notin\mathbb{Z}\quad(1\le i\le n),
\\[4pt]
\rho_i-\rho_j\notin\mathbb{Z}\quad(1\le i<j\le 3),
\quad&\rho_i\notin\mathbb{Z}\quad(1\le i\le 3). 
\end{array}
\end{equation}
In the cases of type $({\rm II})_{2n}$ and $({\rm III})_{2n+1}$, 
the following canonical forms of the Okubo systems are proposed 
by Haraoka \cite{YHEq}. 
\begin{Theorem}[\cite{YHEq}]
Canonical forms of the 
Okubo systems of type $({\rm II})_{2n}$ $(m=n)$ and 
$({\rm III})_{2n+1}$ $(m=n+1)$ are 
given by the matrices $K=(K_{ij})_{ij}$ and $L=(L_{ij})_{ij}$ with the following entries respectively $:$ 
\begin{eqnarray}\label{eq:CII}
({\rm II})_{2n}:&&
\begin{array}{rll}
K_{ij}&=(\beta_j-\rho_1)\ds
\frac{\prod_{k\neq i,\, 1\leq k\leq n}(\alpha_k+\beta_j-\rho_1-\rho_2)}
{\prod_{k\neq j,\,1\leq k\leq n}(\beta_j-\beta_k)}
\quad(1\le i,j\le n)
\\
L_{ij}&=(\alpha_j-\rho_1)\ds
\frac{\prod_{k\neq i,\, 1\leq k\leq n}(\alpha_j+\beta_k-\rho_1-\rho_2)}
{\prod_{k\neq j,\, 1\leq k\leq n}(\alpha_j-\alpha_k)}
\quad(1\le i,j\le n),
\end{array}
\\[4pt]
\label{eq:CIII}
({\rm III})_{2n+1}:&&
\begin{array}{rll}
K_{ij}&=\ds
\frac{\prod_{k\neq i,\, 1\leq k\leq n+1}(\alpha_k+\beta_j-\rho_1-\rho_2)}
{\prod_{k\neq j,\, 1\leq k\leq n}(\beta_j-\beta_k)}\ \
&(1\le i\le n+1;\ 1\le j\le n)
\\[16pt]
L_{ij}&=(\alpha_j-\rho_1)(\alpha_j-\rho_2)\ds
\frac{\prod_{k\neq i,\, 1\leq k\leq n}(\alpha_j+\beta_k-\rho_1-\rho_2)}
{\prod_{k\neq j,\, 1\leq k\leq n+1}(\alpha_j-\alpha_k)}\ \
&(1\le i\le n;\ 1\le j\le n+1). 
\end{array}
\end{eqnarray}
\end{Theorem}

\noindent
To be more precise, the canonical form of Haraoka \cite{YHEq} 
is the conjugation of our $A$ by the diagonal matrix
\begin{equation}
{\rm diag}(a_1^{-1},\ldots,a_m^{-1},b_1^{-1},\ldots,b_n^{-1});
\quad
a_i=\prod_{k=1 \atop k\ne i}^{m}(\alpha_i-\alpha_k),\quad
b_i=\prod_{k=1 \atop k\ne i}^{n}(\beta_i-\beta_k).
\end{equation}

\subsection{Main theorems}
For each Okubo system, the monodromy matrices for 
the canonical solution matrix are expressed in the form 
\eqref{eq:Mk} by using the connection coefficients.
Our main results are the explicit formulas for the connection coefficients
and the monodromy matrices for the canonical solution matrices 
of types ${\rm I}$, ${\rm I}^*$, ${\rm II}$ and ${\rm III}$.
In what follows, we use the notation 
$(t_i-t_j)^{\alpha}=\exp(\alpha\log(t_i-t_j))$ 
$(\alpha\in\mathbb{C})$ with the convention
of arguments such that 
$\theta_j-\pi<{\rm arg}(t_i-t_j)<\theta_j$ 
for $i<j$, and 
$\theta_j<{\rm arg}(t_i-t_j)<\theta_j+\pi$ 
for $i>j$, using $\theta_j=\arg(p_0-t_j)$ as specified in Section 1.1.
\par\medskip
\noindent{\bf Case ${\rm I}$}:\,
For the Okubo system of type $({\rm I})_n$, the canonical solution matrix
\begin{equation}
\Psi(x)=(\psi_1(x),\ldots,\psi_{n-1}(x),\psi_{n}(x))
\in {\rm Mat}(n;\mathcal{O}(\widetilde{\mathcal{D}}))
\end{equation}
is a multivalued holomorphic solution matrix on 
$\mathcal{D}=\mathbb{P}^{1}\backslash\{t_1,t_2,\infty\}$ 
characterized by the following conditions: 
\begin{equation}
\begin{split}
\psi_j(x)&=(x-t_1)^{\alpha_j}({\bs e}_j+O(x-t_1))\quad (j=1,\ldots, n)
\end{split}
\end{equation}
around $x=t_1$, and 
\begin{equation}
\begin{split}
\psi_{n}(x)&=(x-t_2)^{\alpha_n}({\bs e}_{n}+O(x-t_2))\quad 
\end{split}
\end{equation}
around $x=t_2$, 
where ${\bs e}_j=(\delta_{ij})_{i=1}^{n}$ denotes the $j$-th unit column vector. 

According to the representation \eqref{eq:Mk} of the previous section, 
the monodromy matrices $M_1,M_2$ are expressed in the form 
\begin{eqnarray}\label{eq:MI}
M_1&=&\left(
\begin{array}{cc}
e_1 & (e_1-1)C \\
0    & 1
\end{array}\right),
\quad 
e_1={\rm diag}(e(\alpha_1),\ldots,e(\alpha_{n-1})),
\nonumber\\
M_2&=&
\left(
\begin{array}{cc}
 I_{n-1} & \,0 \\
 (e_2-1)D & e_2
\end{array}
\right),\quad
e_2=e(\alpha_n), 
\end{eqnarray}
in terms of the connection matrices $C=C^{(12)}$, 
$D=C^{(21)}$ in \eqref{eq:Ckj}.  
These connection matrices $C=(C_{i})_{i}$ and $D=(D_{j})_{j}$ are defined 
as 
\begin{equation}
\psi_{n}(x)=\sum_{i=1}^{n-1}\psi_{i}(x) C_{i}+ h_{n}(x),
\quad h_{n}(x)\in\mathcal{O}_{t_1}
\end{equation}
around $x=t_1$ and 
\begin{equation}
\psi_{j}(x)=\psi_{n}(x) D_{j}+ h_j(x),
\quad h_j(x)\in\mathcal{O}_{t_2}\quad(j=1,\ldots,n-1)
\end{equation}
around $x=t_2$. 
As to the monodromy around $x=\infty$, we remark that 
\begin{equation}
\quad
M_\infty^{-1}=M_1M_2
\sim {\rm diag}(f_1,f_2,\ldots,f_n), 
\end{equation}
where $f_i=e(\rho_i)\quad(i=1,2,\ldots,n)$. 
The connection matrices $C$ and $D$ are given by the following theorem:

\begin{Theorem}\label{mainI}
The monodromy matrices $M_1$, $M_2$ 
for the canonical solution matrix $\Psi(x)$ of type $({\rm I})_n$  
is expressed as \eqref{eq:MI} in terms of the connection matrices 
$C=(C_{i})_{i}$ and $D=(D_{j})_{j}$ determined as follows $:$ 
\begin{equation}
\begin{split}
C_i&=(-1)^{n}\,e(\tfrac{1}{2}(\rho_2-\alpha_i-\alpha_n))
\frac{(t_1-t_2)^{\rho_2-\alpha_i}}{(t_2-t_1)^{\rho_2-\alpha_n}}
\Gamma(-\alpha_i)\Gamma(\alpha_n+1)
\frac{\prod_{k\neq i,\,1\leq k\leq n-1}\Gamma(1+\alpha_{k}-\alpha_i)}{\prod_{k=1}^n\Gamma(1+\rho_{k}-\alpha_i)},\\
D_{j}&=
e(\tfrac{1}{2}(-\rho_2+\alpha_j+\alpha_n))
\frac{(t_2-t_1)^{\rho_2-\alpha_n}}{(t_1-t_2)^{\rho_2-\alpha_j}}
\Gamma(1+\alpha_j)\Gamma(-\alpha_n)
\frac{\prod_{k\neq j,\,1\leq k\leq n-1}\Gamma(\alpha_j-\alpha_{k})}
{\prod_{k=1}^n\Gamma(\alpha_j-\rho_{k})}.
\end{split}
\end{equation}
\end{Theorem}

\noindent{\bf Case ${\rm I}^*$}:
For the Okubo system of type $({\rm I}^*)_n$, the canonical solution matrix
\begin{equation}
\Psi(x)=(\psi_1(x),\psi_2(x),\ldots,\psi_{n}(x))
\in {\rm Mat}(n;\mathcal{O}(\widetilde{\mathcal{D}}))
\end{equation}
is a multivalued holomorphic solution matrix on 
$\mathcal{D}=\mathbb{P}^{1}\backslash\{t_1,t_2,\ldots,t_n\infty\}$ 
characterized by the following conditions: 
\begin{equation}
\begin{split}
\psi_j(x)&=(x-t_j)^{\alpha_j}({\bs e}_j+O(x-t_j))\quad (j=1,\ldots, n)
\end{split}
\end{equation}
around $x=t_j$.

The monodromy matrices $M_1,M_2,\ldots,M_n$ are expressed in the form 
\begin{equation}
M_i=\begin{pmatrix}
1&&O&&\\
(e_i-1)C_{i1}&\cdots&e_i&\cdots &(e_i-1)C_{in}\\
&&O&&1
\end{pmatrix}
\quad (i=1,\ldots,n).
\end{equation}
These connection matrices $C_{ij}$ are defined as 
\begin{equation}
\psi_{j}(x)=\psi_{i}(x) C_{ij}+ h_{ij}(x),
\quad h_{ij}(x)\in\mathcal{O}_{t_i}
\quad(j=1,\ldots,n).
\end{equation}

\begin{Theorem}\label{mainI*}
The monodromy matrices $M_1,M_2,\ldots,M_n$ 
for the canonical solution matrix $\Psi(x)$ of type $({\rm I}^\ast)_n$
is expressed as \eqref{eq:MI} in terms of the connection matrices 
$C_{ij}$ determined as follows $:$ 
\begin{equation}
C_{ij}=
\begin{cases}
&\ds e(\tfrac{-\rho_1}{2})
\frac{\prod_{k\neq i,\,1\leq k\leq n}(t_i-t_k)^{\alpha_k-\rho_1}}
{\prod_{k\neq j,\,1\leq k\leq n}(t_j-t_k)^{\alpha_k-\rho_1}}
\frac{\Gamma(-\alpha_i)\Gamma(\alpha_j+1)}
{\Gamma(\alpha_j-\rho_1)\Gamma(1+\rho_1-\alpha_i)}
\quad (i<j),\\[10pt]
&\ds e(\tfrac{\rho_1}{2})
\frac{\prod_{k\neq i\,1\leq k\leq n}(t_i-t_k)^{\alpha_k-\rho_1}}
{\prod_{k\neq j,\,1\leq k\leq n}(t_j-t_k)^{\alpha_k-\rho_1}}
\frac{\Gamma(-\alpha_i)\Gamma(\alpha_j+1)}
{\Gamma(\alpha_j-\rho_1)\Gamma(1+\rho_1-\alpha_i)}
\quad(i>j).
\end{cases}
\end{equation}
\end{Theorem}

\par\medskip
\noindent{\bf Cases ${\rm II}$ and ${\rm III}$}:
For the Okubo system type $({\rm II})_{2n}$ and $({\rm III})_{2n+1}$,   
the canonical solution matrix 
\begin{equation}
\Psi(x)=(\psi_1(x),\ldots,\psi_m(x),\psi_{m+1}(x),\ldots,
\psi_{m+n}(x))
\in {\rm Mat}(m+n;\mathcal{O}(\widetilde{\mathcal{D}}))
\end{equation}
is a multivalued holomorphic solution matrix on 
$\mathcal{D}=\mathbb{P}^{1}\backslash\{t_1,t_2,\infty\}$ 
characterized by the following conditions: 
\begin{equation}
\begin{split}
\psi_j(x)&=(x-t_1)^{\alpha_j}({\bs e}_j+O(x-t_1))\quad (j=1,\ldots, m)
\end{split}
\end{equation}
around $x=t_1$, and 
\begin{equation}
\begin{split}
\psi_{m+j}(x)&=(x-t_2)^{\beta_j}({\bs e}_{n+j}+O(x-t_2))\quad (j=1,\ldots, n)
\end{split}
\end{equation}
around $x=t_2$.

According to the representation \eqref{eq:Mk} of the previous section, 
the monodromy matrices $M_1,M_2$ are expressed in the form 
\begin{eqnarray}\label{eq:MCD}
M_1&=&\left(
\begin{array}{cc}
e_1 & (e_1-1)C \\
0    & I_n
\end{array}\right),
\quad 
e_1={\rm diag}(e(\alpha_1),\ldots,e(\alpha_m)),
\nonumber\\
M_2&=&
\left(
\begin{array}{cc}
 I_m & \,0 \\
 (e_2-1)D & e_2
\end{array}
\right),\quad
e_2={\rm diag}(e(\beta_1),\ldots,e(\beta_n)), 
\end{eqnarray}
in terms of the connection matrices $C=C^{(12)}$, 
$D=C^{(21)}$ in \eqref{eq:Ckj}.  
These connection matrices $C=(C_{ij})_{ij}$ and $D=(D_{ij})_{ij}$ are defined 
as 
\begin{equation}
\psi_{m+j}(x)=\sum_{i=1}^{n}\psi_{i}(x) C_{ij}+ h_{m+j}(x),
\quad h_{m+j}(x)\in\mathcal{O}_{t_1}\quad (j=1,\ldots,m)
\end{equation}
around $x=t_1$ and 
\begin{equation}
\psi_{j}(x)=\sum_{i=1}^{n}\psi_{n+i}(x) D_{ij}+ h_j(x),
\quad h_j(x)\in\mathcal{O}_{t_2}\quad(j=1,\ldots,m)
\end{equation}
around $x=t_2$. 
As to the monodromy around $x=\infty$, we remark that 
\begin{equation}
\quad
M_\infty^{-1}=M_1M_2
\sim {\rm diag}(\overbrace{f_1,\ldots, f_1}^{n},\overbrace{f_2,\ldots,f_2}^{m-1},f_3), 
\end{equation}
where $f_i=e(\rho_i)\quad(i=1,2,3)$. 

\begin{Theorem}\label{mainII,III}
The monodromy matrices $M_1$, $M_2$ 
for the canonical solution matrices $\Psi(x)$ of types $({\rm II})_{2n}$
and $({\rm III})_{2n+1}$ 
are expressed as \eqref{eq:MCD} in terms of the connection matrices 
$C=(C_{ij})_{ij}$ and $D=(D_{ij})_{ij}$ determined as follows $:$ 
\newline
\begin{eqnarray}
&&({\rm II})_{2n}:\nonumber\\
&&\begin{array}{rll}
C_{ij}&=
(-1)^{n-1}\ds{
\frac{(t_1-t_2)^{\rho_3-\alpha_i}}{(t_2-t_1)^{\rho_3-\beta_j}}e(\tfrac{1}{2}(\rho_3-\alpha_i-\beta_j))
\frac{\Gamma(\beta_j+1)\Gamma(-\alpha_i)}{\Gamma(1+\rho_1-\alpha_i)\Gamma(\beta_j-\rho_1)}}\\
&\ds{\frac{\prod_{k\neq i,\,1\leq k\leq n}\Gamma(1+\alpha_{k}-\alpha_i)}
{\prod_{k\neq j,\,1\leq k\leq n}\Gamma(1+\rho_1+\rho_2-\alpha_1-\beta_{k})}
\frac{\prod_{k\neq j,\,1\leq k\leq n}\Gamma(\beta_j-\beta_{k})}
{\prod_{k\neq i,\,1\leq k\leq n}\Gamma(\beta_j+\alpha_{k}-\rho_1-\rho_2)}}
\quad (1\leq i,j\leq n),\\[12pt]
D_{ij}
&=
(-1)^{n-1}
\ds{\frac{(t_2-t_1)^{\rho_3-\beta_i}}{(t_1-t_2)^{\rho_3-\alpha_j}}
e(\tfrac{1}{2}(\alpha_j+\beta_i-\rho_3))
\frac{\Gamma(-\beta_{i})\Gamma(\alpha_j+1)}{\Gamma(\alpha_j-\rho_1)\Gamma(1+\rho_1-\beta_i)}}\\
&
\ds{\frac{\prod_{k\neq j,\,1\leq k\leq n}\Gamma(\alpha_j-\alpha_{k})}
{\prod_{k\neq i,\,1\leq k\leq n}\Gamma(\alpha_j+\beta_{k}-\rho_1-\rho_2)}
\frac{\prod_{k\neq i,\,1\leq k\leq n}\Gamma(1+\beta_{k}-\beta_i)}
{\prod_{k\neq j,\,1\leq k\leq n}\Gamma(1+\rho_1+\rho_2-\alpha_k-\beta_i)}}
\quad (1\leq i,j\leq n),
\end{array}
\end{eqnarray}
\begin{eqnarray}
&&({\rm III})_{2n+1}:\\
&&\begin{array}{rll}
 C_{ij}&=
(-1)^{n}\ds{e(\tfrac{1}{2}(\rho_3-\alpha_i-\beta_j))
\frac{(t_1-t_2)^{\rho_3-\alpha_i}}{(t_2-t_1)^{\rho_3-\beta_j}}
\frac{\Gamma(\beta_j+1)\Gamma(-\alpha_i)}{\Gamma(\alpha_i-\rho_1)\Gamma(\alpha_i-\rho_2)}}\\
&\ds{\frac{\prod_{k\neq i,\,1\leq k\leq n+1}\Gamma(1+\alpha_{k}-\alpha_i)}
{\prod_{k\neq j,\,1\leq k\leq n}\Gamma(1+\rho_1+\rho_2-\alpha_1-\beta_{k})}
\frac{\prod_{k\neq j,\,1\leq k\leq n}\Gamma(\beta_j-\beta_{k})}
{\prod_{k\neq i,\,1\leq k\leq n+1}\Gamma(\beta_j+\alpha_{k}-\rho_1-\rho_2)}}
\,
\left(\begin{array}{ll}
1\leq i\leq n+1\\
1\leq j\leq n
\end{array}\right),\\[12pt]
D_{ij}
&=
(-1)^{n-1}\ds{e(\tfrac{1}{2}(\alpha_j+\beta_i-\rho_3))
\frac{(t_2-t_1)^{\rho_3-\beta_i}}{(t_1-t_2)^{\rho_3-\alpha_j}}
\frac{\Gamma(-\beta_{i})\Gamma(\alpha_j+1)}{\Gamma(1+\rho_1-\alpha_j)\Gamma(1+\rho_2-\alpha_j)}}\\
&
\ds{\frac{\prod_{k\neq j,\,1\leq k\leq n+1}\Gamma(\alpha_j-\alpha_{k})}
{\prod_{k\neq i,\,1\leq k\leq n}\Gamma(\alpha_j+\beta_{k}-\rho_1-\rho_2)}
\frac{\prod_{k\neq i,\,1\leq k\leq n}\Gamma(1+\beta_{k}-\beta_i)}
{\prod_{k\neq j,\,1\leq k\leq n+1}\Gamma(1+\rho_1+\rho_2-\alpha_k-\beta_i)}}
\,\left(\begin{array}{ll}
1\leq i\leq n\\
1\leq j\leq n+1
\end{array}\right).
\end{array}
\end{eqnarray}
\end{Theorem}

In the following sections, we prove these theorems by the method of middle convolutions.

\section{Middle convolutions for Schlesinger systems}

In this section, we briefly recall the definitions of the Katz
operations (the addition and the 
middle convolution) for a Schlesinger system and  
its monodromy representation (\cite{DR2007}).

\subsection{Addition}

For each $r$-tuple of matrices 
${\bs A}=(A_1,\ldots,A_r)\in {\rm Mat}(n;\mathbb{C})^r$, 
we consider the Schlesinger system 
\begin{equation}\label{eq:Sch2}
\frac{d}{dx} Y=\left(\sum_{k=1}^{r}\frac{A_k}{x-t_k}\right)Y, 
\quad A_k\in{\rm Mat}(n;\mathbb{C})\ \ (k=1,\ldots,r), 
\end{equation}
of ordinary differential equations on 
$\mathcal{D}=\mathbb{P}^1\backslash \{t_1,\ldots,t_r,\infty\}$. 
We denote by $A_\infty=-A_1-\cdots-A_r$ the residue matrix at
$x=\infty$.  
For an $r$-tuple 
${\bs a}=(a_1,\ldots, a_r)\in \mathbb{C}^r$, 
the addition ${\rm add}_{\bs a}({\bs A})$ of ${\bs A}$ by ${\bs a}$
is defined simply as
\begin{equation}
{\rm add}_{\bs a}({\bs A})=(A_1+a_1,\ldots,A_r+a_r).
\end{equation}
The corresponding Schlesinger system is given by 
\begin{equation}\label{eq:addSch}
\frac{d}{dx} Z=\left(\sum_{k=1}^{r}\frac{A_k+a_k}{x-t_k}\right)Z, 
\quad A_k\in{\rm Mat}(n;\mathbb{C})\ \ (k=1,\ldots,r). 
\end{equation}

Let  $Y(x)$ be a fundamental solution matrix 
of the Schlesinger system \eqref{eq:Sch2} associated with 
an $r$-tuple of matrices 
$\bs{A}\in{\rm Mat}(n;\mathbb{C})^r$. 
Then the monodromy matrices $M_k$ ($k=1,\ldots,r$) 
defined as 
\begin{equation}
\gamma_{k}\!\cdot\! Y(x)=Y(x)M_k
\end{equation}
give rise to an $r$-tuple ${\bs M}=(M_1,\ldots,M_r)
\in{\rm GL}(n;\mathbb{C})^r$ of invertible matrices.
The addition for 
an $r$-tuple ${\bs M}=(M_1,\ldots,M_r)$ is defined by 
\begin{equation}\label{addMon}
{\rm Add}_{{\bs \lambda}}({\bs M})=(\lambda_1 M_1, \ldots, \lambda_r M_r ) 
\end{equation}
for each ${\bs \lambda }=(\lambda_1,\ldots,\lambda_r)\in(\mathbb{C}^\ast)^r$, where 
$\mathbb{C}^\ast=\mathbb{C}\backslash\{0\}$.

We remark that the Schlesinger system \eqref{eq:addSch} is obtained
from \eqref{eq:Sch2} by the operation
\begin{equation}
Z=\prod_{k=1}^r(x-t_k)^{a_k}Y.
\end{equation}
Then 
the monodromy matrices for the fundamental solution 
matrix $Z(x)=\prod_{k=1}^r(x-t_k)^{a_k}Y(x)$ of \eqref{eq:addSch}
is 
given by \eqref{addMon} with $\lambda_k=e(a_k)$ 
$(k=1,\ldots,r)$. 

\subsection{Middle convolution of a Schlesinger system}
The middle convolution 
${\rm mc}_{\mu}({\bs A})$ of ${\bs A}$ with parameter $\mu\in\mathbb{C}^\ast$ 
is defined through three steps (\cite{DR2007}).

\par\medskip\noindent
{\bf Step 1}  (Convolution):\ \ 
We first define the {\em convolution}  ${\rm c}_{\mu}({\bs A})$ of ${\bs A}$ with 
parameter $\mu$ as the $r$-tuple of 
matrices ${\bs B}=(B_1,\ldots,B_r)\in {\rm Mat}(nr;\mathbb{C})$ by setting 
\begin{equation}
B_k=
\begin{pmatrix}
0   & & \ldots & & 0 \\
    & \ddots & & & \\
A_1 & \ldots & A_k+\mu &\ldots &A_r\\
    & & & \ddots & \\
0   & & \ldots & & 0
\end{pmatrix}
\in {\rm Mat}(nr;\mathbb{C})\qquad(k=1,\ldots,r)
\end{equation}
where $B_k$ is zero outside the $k$-th block row.
The Schlesinger system associated with $\bs B$
can be written in the form of an Okubo system
\begin{equation}\label{eq:convolution}
(x-T)\frac{d}{dx}Z=BZ, 
\end{equation} 
where 
$T={\rm diag}(t_1I_{n},\ldots,t_rI_{n})$ 
and $B=B_1+\cdots+B_r$.  

\par\medskip\noindent

\medskip

\noindent{\bf Step 2} ($K$-Reduction):\ \
Setting $n_k={\rm rank}\, A_k$,
we decompose $A_k$ $(k=1,\ldots,r)$ as 
$A_k=P_kQ_k$ with two matrices 
$P_k\in {\rm Mat}(n,n_k;\mathbb{C})$, 
$Q_k\in {\rm Mat}(n_k,n;\mathbb{C})$, 
 and take 
the block matrix 
\begin{equation}\label{QSA}
Q=
\begin{pmatrix}
Q_1 & & \\
 & \ddots & \\
 & & Q_r 
\end{pmatrix}
\in\mbox{\rm Mat}(\widetilde{n},nr;\mathbb{C})\qquad 
(\widetilde{n}=\sum_{k=1}^rn_k ).
\end{equation}
Since the $(i,j)$-block of $B$ is given by 
$B_{ij}=A_{j}+\mu\delta_{ij}=P_jQ_j+\mu\delta_{ij}$, we have 
$Q_iB_{ij}=(Q_iP_j+\mu\delta_{ij})Q_j$, namely
\begin{equation}
QB=\widetilde{B}Q,\quad 
\widetilde{B}=\Big(Q_{i}P_j+\mu\delta_{ij}\Big)_{i,j=1}^{r}.  
\end{equation}
Hence, the system \eqref{eq:convolution}
gives rise to the Okubo system
\begin{equation}
(x-T)\frac{d\ }{dx}\widetilde{Z}=\widetilde{B}\widetilde{Z}
\end{equation}
for $\widetilde{Z}=QZ$. 
Equivalently, we obtain the Schlesinger system 
\begin{equation}\label{eq:K-part}
\frac{d}{dx}\widetilde{Z}=\left(
\sum_{k=1}^r\frac{\widetilde{B}_k}{x-t_k}
\right)\widetilde{Z}
\end{equation}
corresponding to $\widetilde{\bs B}=(\widetilde{B}_1,\ldots,
\widetilde{B}_r)$, 
where
\begin{equation}
\widetilde{B}_k
=
\begin{pmatrix}
0   & & \ldots & & 0 \\
    & \ddots & & & \\
Q_kP_1 & \ldots & Q_kP_k+\mu &\ldots &Q_kP_r\\
    & & & \ddots & \\
0   & & \ldots & & 0
\end{pmatrix}\quad(k=1,\ldots,r).  
\end{equation}
\par\medskip\noindent
{\bf Step 3} ($L$-Reduction):\ \  
Setting $m={\rm rank}\,\widetilde{B}$, 
we further decompose $\widetilde{B}$ as 
$\widetilde B=P_0Q_0$ by two matrices 
$P_0\in {\rm Mat}(\widetilde{n}, m;\mathbb{C})$, 
$Q_0\in {\rm Mat}(m,\widetilde{n};\mathbb{C})$. 
Since $\widetilde{B}_k$ is expressed as $\widetilde{B}_k
=E_k\widetilde{B}$, $E_k={\rm diag}(0.\ldots,I_{n_k},\ldots,0)$, 
we have
\begin{equation}\label{Q0S0A}
Q_0\widetilde{B}_k=Q_0E_k\widetilde{B}=
Q_0E_kP_0Q_0.  
\end{equation}
Hence, multiplying \eqref{eq:K-part} by $Q_0$  
we obtain 
\begin{equation}
\frac{d}{dx}Q_0\widetilde{Z}=\left(
\sum_{k=1}^r\frac{Q_0E_kP_0}{x-t_k}
\right)Q_0\widetilde{Z}. 
\end{equation}
Setting $\widehat{Z}=Q_0\widetilde{Z}$, we obtain the 
Schlesinger system 
\begin{equation}\label{eq:L-part}
\frac{d}{dx}\widehat{Z}=\left(
\sum_{k=1}^r \frac{\widehat{B}_k}{x-t_k}
\right)\widehat{Z},
\quad \widehat{B}_k=Q_0E_kP_0\quad(k=1,\ldots,r)
\end{equation}
associated with 
${\rm mc}_{\mu}({\bs A})=(\widehat{B}_1,\ldots,\widehat{B}_r)$. 
We call the system \eqref{eq:L-part} the 
{\em middle convolution}
of \eqref{eq:Sch2} with parameter $\mu$.
We note that these matrices $\widehat B_i$ $(i=1,\ldots,r)$ 
are realizations of the linear transformations on 
$\mathbb{C}^{nr}/\big((\bigoplus_{k=1}^r{\rm Ker}\,A_k)\oplus{\rm Ker}\,B\big)$ induced from $B_i$.

\subsection{Middle convolution of a monodromy representation}

We next recall the middle convolution of 
monodromy matrices.  

Let  $Y(x)$ be a fundamental solution matrix 
of the Schlesinger system \eqref{eq:Sch2} associated with 
an $r$-tuple of matrices 
$\bs{A}\in{\rm Mat}(n;\mathbb{C})^r$. 
Then the monodromy matrices $M_k$ ($k=1,\ldots,r$) 
defined as 
\begin{equation}\label{eq:gYM}
\gamma_{k}\!\cdot\! Y(x)=Y(x)M_k
\end{equation}
give rise to an $r$-tuple ${\bs M}=(M_1,\ldots,M_r)
\in{\rm GL}(n;\mathbb{C})^r$ of invertible matrices. 
The multiplicative middle convolution 
${\rm MC}_\lambda(\bs{M})$ we are going to explain below 
provides a way to  
construct the monodromy matrices for a certain 
fundamental solution matrix of the Schlesinger system 
associated with the middle convolution 
${\rm mc}_{\mu}(\bs A)$ with $\lambda=e(\mu)$. 

\par\medskip
Let $\bs{M}=(M_1,\ldots,M_r)\in{\rm GL}(n;\mathbb{C})^r$ 
be an arbitrary $r$-tuple of invertible matrices.  
The multiplicative middle convolution 
${\rm MC}_{\lambda}({\bs M})$ of $\bs{M}$ with parameter 
$\lambda\in\mathbb{C}^\ast$ is constructed 
through three steps.  

\par\medskip\noindent 
{\bf Step 1} (Convolution):\ \ 
We define an $r$-tuple of invertible matrices 
${\bs N}=(N_1,\ldots,N_r)\in{\rm GL}(nr;\mathbb{C})^r$ 
as follows: 
\begin{equation}\label{mon:convolution}
N_k=
\begin{pmatrix}
1   & & \ldots & & 0 \\
    & \ddots & & & \\
\lambda(M_1-1) & \ldots & \lambda M_k &\ldots &M_r-1\\
    & & & \ddots & \\
0   & & \ldots & & 1
\end{pmatrix}
\in {\rm GL}(nr;\mathbb{C}).  
\end{equation}
Then we call $\bs N$ the {\em 
convolution} of $\bs M$ with parameter $\lambda$, and denote it by 
${\rm C}_{\lambda}(\bs M)$.  
As we will see later, ${\rm C}_{\lambda}(\bs M)$ represents 
the $r$-tuple of monodromy matrices for a fundamental 
solution matrix of the Schlesinger system 
${\rm c}_{\mu}({\bs A})$.

\par\medskip\noindent
{\bf Step 2} ($\mathcal{K}$-Reduction):\ \ 
Setting $n_k={\rm rank}\,(M_k-1)$, 
we decompose $M_k-1\,(k=1,\ldots,r)$ as $M_k-1 =\cP_k\cQ_k$ with two
matrices 
$\cP_k\in{\rm Mat}(n,n_k;\mathbb{C}),\, \cQ_k\in{\rm Mat}(n_k,n;\mathbb{C})$. 
Taking a right inverse $\cS_k\in {\rm Mat}(n,n_k;\mathbb{C})$ of $\cQ_k$ 
for each $k=1,\ldots,r$ so that $\cQ_k\cS_k=I_{n_k}$, 
we define the block matrices $\cQ$ and 
$\cS$  by 
\begin{equation}\label{QS}
\cQ=\begin{pmatrix}
\cQ_1 &&\\
&\ddots &\\
&& \cQ_r
\end{pmatrix}
\in\mbox{\rm Mat}(\widetilde{n}, nr;\mathbb{C}),
\quad
\cS=\begin{pmatrix}
\cS_1 &&\\
&\ddots &\\
&& \cS_r
\end{pmatrix}
\in\mbox{\rm Mat}(nr,\widetilde{n};\mathbb{C}).  
\end{equation}
where $\widetilde{n}=\sum_{i=1}^{r} n_i$. 
Since the $(k,j)$-block of $N_k-1$ is given as 
\begin{equation}
(N_k-1)_{kj}=
\begin{cases}
\ \lambda(M_j-1)=\lambda \cP_k\cQ_k  &(1\le j<k)\\
\ \lambda M_k-1=\lambda \cP_k\cQ_k+(\lambda-1) &(j=k)\\
\ M_j-1=\cP_k\cQ_k &(k<j\le r), 
\end{cases}
\end{equation}
by $\cQ_k\cS_k=1$ we obtain
$\cQ(N_k-1)\cS=\widetilde{N}_{k}-1$, i.e.\ $\cQ N_k\cS=\widetilde{N}_k$, 
where 
\begin{equation}\label{mon:convolution2}
\widetilde{N}_k=
\begin{pmatrix}
1   & & \ldots & & 0 \\
    & \ddots & & & \\
\lambda \cQ_k\cP_1 & \ldots & \lambda (\cQ_k\cP_k+1)&\ldots &\cQ_k\cP_r\\
    & & & \ddots & \\
0   & & \ldots & & 1
\end{pmatrix}
\in {\rm GL}(\widetilde{n};\mathbb{C}), 
\end{equation}
which gives a ${\mathcal K}$-reduction of the middle convolution 
$\bs{\widetilde{N}}=(\widetilde{N}_1,\ldots,\widetilde{N}_r)$.
\par\medskip
\noindent{\bf Step 3}  (${\cal L}$-reduction):\ \
To construct monodromy matrices corresponding to the middle convolution
${\rm MC}_{\lambda}(\bs{M})$, 
we set $\widetilde{N}_0=\widetilde{N}_1\cdots \widetilde{N}_r$, 
and decompose $\widetilde{N}_0-1$, 
as $\widetilde{N}_0-1=\cP_0\cQ_0$ 
by two matrices 
$\cP_0\in {\rm Mat}(\widetilde{n}, m;\mathbb{C})$, 
$\cQ_0\in {\rm Mat}(m,\widetilde{n};\mathbb{C})$
with $m={\rm rank}\,(\widetilde{N}_0-1)$, and 
take a right inverse $\cS_0\in{\rm Mat}(\widetilde{n},m;\mathbb{C})$ 
so that $\cQ_0\cS_0=1$.  
Then we obtain an $r$-tuple of monodromy matrices 
$\widehat{\bs{N}}(\widehat{N}_1,\ldots,\widehat{N}_r)$ 
as
\begin{equation}\label{Q0S0}
\widehat{N}_k=\cQ_0 \widetilde{N}_k \cS_0\qquad(k=1,\ldots,r).  
\end{equation}
We call the $r$-tuple $\widehat{\bs{N}}$ the 
{\em middle convolution}
of $\bs{M}$ with parameter $\lambda$, and denote by 
${\rm MC}_{\lambda}(\bs M)$.  
We remark that these matrices $\widehat{N}_i$ ($i=1,\ldots,r$) 
are realizations of the linear transformations on 
$\mathbb{C}^{nr}/\big((\bigoplus_{k=1}^{r}{\rm Ker}\,(M_k-1))\oplus
{\rm Ker}\,(N_0-1)\big)$ induced from $N_i$, where 
$N_0=N_1\cdots N_r$. 

\subsection{Fundamental solution matrices}
Let $Y(x)$ be a fundamental solution matrix of the Schlesinger system of \eqref{eq:Sch} and $\bs{M}=(M_1,\ldots,M_r)$ the 
$r$-tuple of invertible matrices defined by the monodromy 
representation as in \eqref{eq:gYM}.  
We summarize below how one can obtain fundamental 
solution matrices whose monodromy representations 
correspond to the convolution ${\rm C}_{\lambda}(\bs{M})$ 
and the middle convolution ${\rm MC}_{\lambda}(\bs{M})$ 
of $\bs{M}$.

Following the construction by \cite{DR2007}, 
for a complex parameter 
$\mu\in\mathbb{C}$ we consider 
the $nr\times n$ block matrices 
\begin{equation}\label{eq:defILY}
I^{\mu}_{L_k}Y(x)=
\bigg(\int_{L_k}(x-u)^{\mu}\,Y(u)
\frac{du}{u-t_i}\bigg)_{i=1}^{r}
\qquad(k=1,\ldots,r), 
\end{equation} 
called the {\em Euler transforms} of $Y(x)$.  
Here $L_k$ denotes the double loop in the 
$u$-plane 
$\mathcal{D}_{x}=\mathbb{C}\backslash\{t_1,\ldots,t_r,x\}$ 
encircling $u=t_k$ and $u=x$ for each $k=1,\ldots,r$.  
We use the symbols $\alpha_\infty,\alpha_1,\ldots,\alpha_r$ and $\alpha_x$ 
for the generators of $\pi_{1}(\mathcal{D}_x,p_0)$ 
encircling $u=\infty, t_1,\ldots,t_r$ and $u=x$ 
in the positive direction 
such that 
$\alpha_\infty\alpha_1\cdots \alpha_r \alpha_x=1$.  
Then the double loops $L_k$ ($k=1,\ldots,r$) are 
expressed as 
$L_k
=L[t_k,x]
=\alpha_k^{-1}\alpha_x^{-1}\alpha_k\alpha_x$
(See Figure 2). 
We also set 
$L_\infty=L[\infty, x]=\alpha_\infty^{-1}\alpha_x^{-1}\alpha_\infty\alpha_x$. 
\begin{figure}[htbp]
\centering
\unitlength=1.2pt
\begin{picture}(80,80)(0,-14)
\put(34,46){\circle*{3}}
\put(73,46){\circle*{3}}
\put(153,46){\circle*{3}}
\put(192,46){\circle*{3}}
\put(105,43){$\cdots$}
\put(30,56){$t_1$}
\put(73,56){$t_2$}
\put(148,56){$t_r$}
\put(190,56){$x$}
\put(110,0){$p_0$}
\put(50,20){$\alpha_1$}
\put(95,35){$\alpha_2$}
\put(123,35){$\alpha_r$}
\put(170,20){$\alpha_x$}
\put(67,36){\vector(-2,1){2}}
\put(95,32){\vector(-2,3){2}}
\put(147,33){\vector(3,2){2}}
\put(175,31){\vector(3,2){2}}
\end{picture}
\includegraphics[scale=0.85, bb=102 -23 300 100]{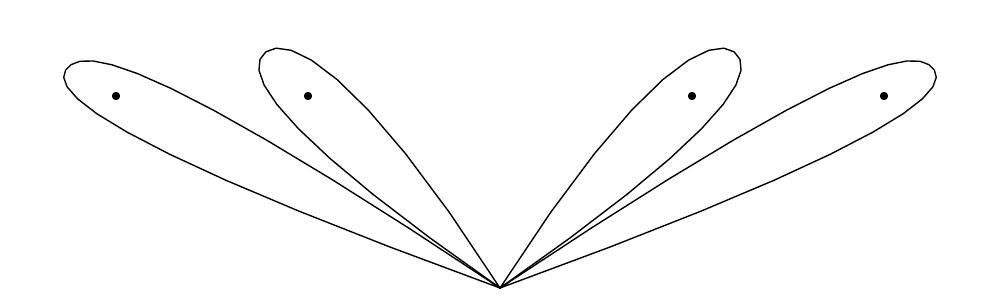}
\caption{}
\end{figure}

We define the $nr\times nr$ block matrix $I^{\mu}Y(x)$ by 
\begin{equation}\label{eq:defIY}
\begin{split}
I^{\mu}Y(x)
&
=
\left(
I^{\mu}_{L_1}Y(x),\ldots,
I^{\mu}_{L_r}Y(x)
\right)
\\[4pt]
&
=
\bigg(
\int_{L_j}(x-u)^\mu\,Y(u)\frac{du}{u-t_i}
\bigg)_{1\le i,j\le r}
\in {\rm Mat}(nr;\mathcal{O}(\widetilde{D})).  
\end{split}
\end{equation}
It is known by \cite{DR2007} that $I^{\mu}Y(x)$ 
is a solution matrix of the Okubo system
\eqref{eq:convolution} associated with 
${\rm c}_{\mu}(\bs{A})$, 
and that 
the monodromy of  
$I^{\mu}Y(x)$ is described as 
\begin{equation}\label{eq:gIY=IYN}
\gamma_k\!\cdot\! I^{\mu}Y(x)=I^{\mu}Y(x)\,
N_k\qquad(k=1,\ldots,r)
\end{equation}
in terms of the convolution 
${\rm C}_\lambda(\bs{M})=(N_1,\ldots,N_r)$ 
of $\bs{M}$ with $\lambda=e(\mu)$. 

\begin{Theorem}\label{thm:IY} 
Suppose that the following three conditions are satisfied. 
\begin{itemize}
\setlength{\itemsep}{-2pt}
\item[$({\rm i})$] 
$\mu\in\mathbb{C}\backslash\mathbb{Z}$. 
\item[$({\rm ii})$]  
For each $k=1,\ldots,r$ and $\infty$,  
any pair of 
distinct eigenvalues of $A_k$
has no integer difference. 
\item[$({\rm iii})$]
$A_k$ $(k=1,\ldots,r)$
and $A_\infty-\mu$ have no 
eigenvalue in $\mathbb{Z}_{>0}$.  
\end{itemize}
Then the 
$nr\times nr$ block matrix
$Z(x)=I^{\mu}Y(x)$ 
is a fundamental solution matrix of the Okubo system
\eqref{eq:convolution} associated with 
${\rm c}_{\mu}(\bs{A})$. 
\end{Theorem} 

We prove that $I^{\mu}Y(x)$ is a fundamental solution matrix. 
The fundamental solution matrix $Y(x)$ of \eqref{eq:Sch2} is written of the
form
\begin{equation}
Y(x)=F^{(k)}(x)\,(x-t_k)^{A_k}C^{(k)},\quad F^{(k)}(x)\in 
{\rm Mat}(n;\mathcal O_{t_k}),\quad F^{(k)}(t_k)=I_n 
\end{equation}
around $x=t_k$ for each $k=1,\ldots,r$, and 
\begin{equation}
Y(x)=F^{(\infty)}(x)\,x^{-A_\infty}C^{(\infty)},\quad F^{(\infty)}(x)\in 
{\rm Mat}(n;\mathcal O_{\infty}),\quad F^{(\infty)}(\infty)=I_n
\end{equation}
around $x=\infty$.  
We investigate 
the local behavior of  
\begin{equation}
I^{\mu}Y(x)_{ij}=
\int_{L_j}\frac{(x-u)^{\mu}}{u-t_i}Y(u)du=
\int_{L_j}\frac{(x-u)^{\mu}}{u-t_i}F^{(j)}(u)(u-t_j)^{A_j}C^{(j)}du \quad
\end{equation}
at $x=t_j$. 
Changing the integration variable by $u=(x-t_j)v+t_j$, 
for $j\ne i$ we have 
\begin{eqnarray}
\int_{L_j}\frac{(x-u)^{\mu}}{u-t_i}Y(u)du
&=&
\int_{L[0,1]}\frac{(1-v)^{\mu}}{(x-t_j)v-(t_i-t_j)}
F^{(j)}((x-t_j)v+t_j)v^{A_j}(x-t_j)^{A_j+\mu+1}C^{(j)}dv
\nonumber
\\
&=&
\int_{L[0,1]}\frac{(1-v)^{\mu}}{t_j-t_i}
F^{(j)}(t_j)v^{A_j}dv\,(x-t_j)^{A_j+\mu+1}C^{(j)}\big(1+O(x-t_j)\big)
\nonumber
\\
&=&
\frac{\widetilde B(A_j+1,\mu+1) }{t_j-t_i}
(x-t_j)^{A_j+\mu+1}C^{(j)}\big(1+O(x-t_j)\big), 
\end{eqnarray}
where 
\begin{equation}
\begin{split}
\widetilde{B}(A,\beta)&=\int_{L[0,1]}u^{A-1}(1-u)^{\beta-1}du
\\
&=(e(A)-1)(e(\beta)-1)B(A,\beta)
\qquad(A\in{\rm Mat}(n;\mathbb{C}),\ \beta\in\mathbb{C})
\end{split}
\end{equation}
denotes the regularized beta function with a matrix argument. 
In this integral, as the base point we take a point 
in the interval $(0,1)$ with $\arg u=\arg (1-u)=0$.  
Similarly we have 
\begin{equation}
\begin{split}
\int_{L_i}\frac{(x-u)^{\mu}}{u-t_i}Y(u)du
&=\widetilde 
B(A_i,\mu+1)(x-t_i)^{A_i+\mu}C^{(i)}\big(1+O(x-t_i)\big)\quad (i=1,\cdots,r),
\\
\int_{L_\infty}\frac{(x-u)^{\mu}}{u-t_i}Y(u)du
&=-e^{-\pi i\mu}\widetilde
B(A_{\infty}-\mu,\mu+1)x^{-A_{\infty} +\mu}C^{(\infty)}
\big(1+O(x^{-1})\big).  
\end{split}
\end{equation}
Since $A_i$ has 
no eigenvalue in $\mathbb{Z}_{>0}$ by the assumption,  
we have $\det \widetilde{B}(A_i,\mu+1)\neq 0$, and hence 
${\rm det}(I^{\mu}Y(x)_{ii})\neq 0$ for $i=1,\ldots,r$.  
Similarly, 
${\rm det}(I^{\mu}_{L_{\infty}}Y(x)_i)\neq 0$ 
for $i=1,\ldots,r$.  
Consider the vector space 
\begin{equation}
\begin{split}
\mathcal{L}&=\bigcap_{i=1}^r{\rm Ker}(N_i-1)={\rm Ker}(N_{\infty}-1)
\\
&=
\{\bs{v}=(M_2\cdots M_rv,\ldots,M_rv,v)^t\in\mathbb{C}^{nr}\ |\ 
v\in{\rm Ker}(\lambda M_1\cdots M_r-1)\},
\end{split}
\end{equation}
where we regard square matrices as linear transformations 
acting on column vectors.
\begin{Lemma}\label{lem:KerIY}
Under the assumption of Theorem \ref{thm:IY}, we have 
${\rm Ker}(I^{\mu}Y(x))\subseteq \mathcal{L}$. 
\end{Lemma} 
\begin{proof}
Assume that a column vector $\bs{v}=(v_1,\ldots,v_r)^{t}
\in\mathbb{C}^{nr}$ belongs to 
${\rm Ker}(I^{\mu}Y(x))$, namely, $I^{\mu}Y(x)\bs{v}=0$.  Then we have 
$(\gamma_k -1)\cdot I^{\mu}Y(x)\bs{v}=0$ for $k=1,\ldots,r$.
This implies that the $k$-th row of $I^{\mu}Y(x)$ satisfies
\begin{equation}
I^{\mu}Y(x)_{kk}\big(\sum_{1\le i<k}e(\mu)(M_i-1)v_i +(e(\mu)M_k-1)v_k+\sum_{k<i\leq
 r}(M_i-1)v_i\big)=0\quad (k=1,\ldots,r).  
\end{equation}
Since ${\rm det}(I^{\mu}Y(x)_{kk})\neq 0$, we have
\begin{equation}
\sum_{1\le i<k}e(\mu)(M_i-1)v_i +(e(\mu)M_k-1)v_k+\sum_{k<i\leq
 r}(M_i-1)v_i=0\quad (k=1,\ldots,r). 
\end{equation}
Hence we find that $\bs{v}\in {\rm Ker}(N_k-1)$ \ \ 
$(k=1,\ldots,r)$. 
\end{proof}
\noindent
To complete the proof, we assume 
$\bs{v}\in {\rm Ker}(I^\mu Y(x))$.  Then by Lemma \ref{lem:KerIY},
$\bs{v}$ is expressed as $\bs{v}=(M_2\cdots M_rv,\ldots,M_rv,v)$ 
for some $v\in{\rm Ker}(\lambda M_1\cdots M_r-1)$.  
In the following we set 
\begin{equation}
I^{\mu}_\alpha Y(x)=
\bigg(\int_{\alpha}(x-u)^{\mu}\,Y(u)
\frac{du}{u-t_i}\bigg)_{i=1}^{r}
\end{equation}
for each $\alpha\in \pi(\mathcal{D}_x,p_0)$. 
Note that 
by $L_k=\alpha_k^{-1}\alpha_x^{-1}\alpha_k\alpha_x$ 
we have 
\begin{equation}
\begin{split}
I^{\mu}_{L_k}Y(x)&=
I^{\mu}_{\alpha_k}Y(x)
(\lambda-1)-
I^{\mu}_{\alpha_x}Y(x)
(M_k-1)\quad(k=1,\ldots,r),
\\
I^{\mu}_{L_\infty}Y(x)&=
I^{\mu}_{\alpha_\infty}Y(x)
(\lambda-1)-
I^{\mu}_{\alpha_x}Y(x)
((\lambda M_1\cdots M_r)^{-1}-1). 
\end{split}
\end{equation}
In this notation $I^{\mu}Y(x)\bs{v}$ is computed as follows: 
\begin{equation}
\begin{split}
I^{\mu}Y(x)\bs{v}
&=
\sum_{k=1}^r I^{\mu}_{L_k}Y(x) 
M_{k+1}\cdots M_rv 
\\
&=
\sum_{k=1}^r 
\big(
I^{\mu}_{\alpha_k}Y(x) (\lambda-1)
-
I^{\mu}_{\alpha_x}Y(x) (M_k-1)
\big)
M_{k+1}\cdots M_rv 
\\
&=
(\lambda-1)\sum_{k=1}^r 
I^{\mu}_{\alpha_k}Y(x)
M_{k+1}\cdots M_rv 
-
\sum_{k=1}^{r}
I^{\mu}_{\alpha_x}Y(x) (M_k-1)
M_{k+1}\cdots M_rv 
\\
&=
(\lambda-1)
I^{\mu}_{\alpha_1\cdots\alpha_r}Y(x)v 
-
I^{\mu}_{\alpha_x}Y(x) (M_1\cdots M_r-1)v
\end{split}
\end{equation}
By 
\begin{equation}
\begin{split}
I^{\mu}_{\alpha_1\cdots\alpha_r}Y(x)
&=
I^{\mu}_{\alpha_\infty^{-1}\alpha_x^{-1}}Y(x)
=
I^{\mu}_{\alpha_\infty^{-1}}Y(x)\lambda^{-1}
+
I^{\mu}_{\alpha_x^{-1}}Y(x)
\\
&=
-
I^{\mu}_{\alpha_\infty}Y(x)M_1\cdots M_r
-
I^{\mu}_{\alpha_x}Y(x)\lambda^{-1}
\end{split}
\end{equation}
we obtain
\begin{equation}
\begin{split}
I^\mu Y(x)\bs{v}&=
-(\lambda-1)I^{\mu}_{\alpha_\infty}Y(x)M_1\cdots M_rv
-I^{\mu}_{\alpha_x}Y(x)(M_1\cdots M_r-\lambda^{-1})v
\\
&=
\left(
-(\lambda-1)I^{\mu}_{\alpha_\infty}Y(x)
+I^{\mu}_{\alpha_x}Y(x)((\lambda M_1\cdots M_r)^{-1}-1)
\right)M_1\cdots M_r v
\\
&=
-I^{\mu}_{L_\infty}Y(x)M_1\cdots M_r v. 
\end{split}
\end{equation}
Since ${\rm det}(I^{\mu}_{L_\infty}Y(x)_i)\neq 0$ 
($i=1,\ldots,r$), 
$I^{\mu}Y(x)\bs{v}=0$ implies $v=0$, and hence 
$\bs{v}=0$. This completes the proof of Theorem \ref{thm:IY}. 

\par\medskip
We set $n_k={\rm rank}\ A_k$ 
for $k=1,\ldots,r$, and $\widetilde{n}=\sum_{i=1}^{r} n_i$. 
\begin{Theorem}\label{thm:IYKpart}
In addition to the conditions 
$(${\rm i}$)$, $(${\rm ii}$)$, $(${\rm iii}$)$ of Theorem 
\ref{thm:IY}, 
suppose that the following condition is satisfied. 
\begin{itemize}
\item[$(${\rm iv}$)$] For each $k=1,\ldots,r$, $A_k$ has
non-integer eigenvalues 
$\alpha_1^{(k)},\ldots,\alpha^{(k)}_{n_k}$ 
where $n_k={\rm rank}\ A_k$. 
\end{itemize}
Using the constant matrices 
$Q$ and $\mathcal{S}$ as in 
\eqref{QSA} and \eqref{QS} respectively, 
we set 
$\widetilde{Z}(x)=QZ(x)\mathcal{S}$, 
$Z(x)=I^\mu Y(x)$. 
Then the 
$\widetilde{n}\times \widetilde{n}$ matrix 
$\widetilde{Z}(x)$ is a fundamental solution matrix of the $K$-reduction
\eqref{eq:K-part} of the convolution, and  
the monodromy matrices for $\widetilde{Z}(x)$ are 
given by the $\mathcal{K}$-reduction $\widetilde{\bs{N}}$ 
of ${\rm C}_\lambda(\bs{M})=\bs{N}$. 
\end{Theorem}
\begin{Theorem}\label{thm:IYLpart}
In addition to the conditions 
$(${\rm i}$)$, $(${\rm ii}$)$, $(${\rm iii}$)$, 
$(${\rm iv}$)$ above,
suppose that the following condition is satisfied. 
\begin{itemize}
\setlength{\itemsep}{-2pt}
\item[$({\rm v})$] $A_\infty-\mu$ has 
non-integer eigenvalues 
$\alpha^{(\infty)}_1-\mu,\ldots,\alpha^{(\infty)}_l-\mu$ 
where 
$l={\rm rank}\,(A_\infty-\mu)$.
\end{itemize}
Using $Q_0$ and $\cS_0$ as in \eqref{Q0S0A} and 
\eqref{Q0S0} respectively, set 
$\widehat{Z}(x)=Q_0\widetilde{Z}(x) \cS_0$.  
Then $\widehat{Z}(x)$ 
is a fundamental solution matrix of the Schlesinger system corresponding to 
the middle convolution \eqref{eq:L-part},
and the monodromy matrices for $\widehat{Z}(x)$ are 
given by the middle convolution 
${\rm MC}_\lambda(\bs{M})=\widehat{\bs{N}}$. 
\end{Theorem}

As we already verified, $QZ(x)$ is a solution matrix of 
the $K$-reduction \eqref{eq:K-part}.  Also, 
under the assumption ${\rm (iv)}$, we have 
${\rm rank}\, A_k={\rm rank}\,(M_k-1)$ for $k=1,\ldots,r$.  
For $\cQ_k$ , $\cS_k$ in \eqref{QS}, we take 
$\cQ'_k$, $\cS'_k$ such that
\begin{equation}
\begin{pmatrix}
\cQ_k \\
\cQ'_k
\end{pmatrix}
\begin{pmatrix}
\cS_k & \cS'_k
\end{pmatrix}=I_n. 
\end{equation}
Since $M_k-1=\cP_k\cQ_k$, we have
\begin{equation}
(M_k-1)
\begin{pmatrix}
\cS_k & \cS'_k
\end{pmatrix}
=(\cP_k,0).
\end{equation}
We consider solution matrix $QZ(x)\widetilde{\cS}$ of \eqref{eq:K-part}
defined by 
\begin{equation}
\widetilde S=
\begin{pmatrix}
\cS & \cS'
\end{pmatrix}
=
\begin{pmatrix}
\cS_1 & & &\cS'_1 & & \\
 & \ddots & & & \ddots & \\
 & & \cS_r & & &\cS'_r
\end{pmatrix}. 
\end{equation}
Noting that 
\begin{equation}
\begin{split}
Z(x)\widetilde{\cS}=I^{\mu}Y(x)\widetilde{\cS}
=(
I^\mu_{L_1}Y(x)\cS_1,\ldots
I^\mu_{L_r}Y(x)\cS_r,
I^\mu_{L_1}Y(x)\cS'_1,\ldots
I^\mu_{L_r}Y(x)\cS'_r
)
\end{split}
\end{equation}
we compute 
the analytic continuation 
by 
$\gamma_k$
as 
\begin{equation}
\gamma_k\!\cdot\!I_{L_j}^{\mu}Y(x)S'_j=
\begin{cases}
I_{L_j}^{\mu}Y(x)S'_j\quad &(j\neq k)\\[2pt]
I_{L_k}^{\mu}Y(x)e(\mu)S'_k\quad &(j=k)
\end{cases}
\end{equation}
by $(M_j-1)\cS'_j=0$. 
Hence we have 
\begin{equation}
\gamma_k\!\cdot\!
QI^{\mu}_{L_k}Y(x)\cS'_k
=QI^{\mu}_{L_k}Y(x)\cS'_k\,e(\mu)\quad(k=1,\ldots,r).
\end{equation}
However, under the assumption ${\rm (iv)}$, 
$e(\mu)$ is {\em not} an eigenvalue of 
$\gamma_k$ on the solution space of 
the $K$-reduction \eqref{eq:K-part}. 
This means that 
$QI^{\mu}_{L_k}Y(x)\cS'_k=0$  $(k=1,\ldots,r)$, namely, 
\begin{equation}
QZ(x)\widetilde{\cS}=(QZ(x)\cS,0).  
\end{equation}
Since ${\rm rank}\,Q=\widetilde{n}$, the matrix 
${\rm rank}\,QZ(x)\widetilde{\cS}=\widetilde{n}$ 
for any regular point $x\in \mathcal{D}$.  
Hence we have ${\rm rank}\, QZ(x)\cS=\widetilde{n}$. 
This completes the proof of Theorem \ref{thm:IYKpart}.  
Theorem \ref{thm:IYLpart} can be proved  in a similar way.

\section{Construction of Okubo systems by middle convolutions}

\subsection{Middle convolution for an Okubo system 
with additions at a singular point}
We consider the operation 
\begin{equation}\label{mcadd}
{\rm add}_{(0,\ldots,\rho,\ldots,0)}\circ {\rm mc}_{-\rho-c}\circ 
{\rm add}_{(0,\ldots,c,\ldots,0)}({\bs A})\qquad(c,\rho\in\mathbb{C})
\end{equation}
for an Okubo system \eqref{eq:Okubo} 
with residue matrices ${\bs A}=(A_1,\ldots,A_r)$, where two additions 
are applied at $x=t_k$ ($k=1,\ldots,r$). 
The Okubo systems of types ${\rm I}$, ${\rm II}$ and ${\rm III}$
are included in the class of the differential systems
obtained from ${\rm (II)}_2$
by a finite iteration of operations in the form \eqref{mcadd}.
We show that the Schlesinger system corresponding to \eqref{mcadd} is 
an Okubo system if the parameters $c$ and $\rho$ are generic in the sense that 
${\rm Ker} (A_k+c)=0$ and ${\rm Ker}(A_{kk}-\rho)=0$. 

\par\medskip
With the notations of \eqref{eq:OkA} and \eqref{eq:Sch}, 
we first decompose the components of 
${\rm add_{0,\ldots,c,\ldots,0}}({\bs A})=
(A_1,\ldots,A_k+c,\ldots,A_r)$ as 
\begin{equation}
\begin{split}
&A_i=P_iQ_i;\quad
P_i=\begin{pmatrix}
0\\
I_{n_i}\\
0
\end{pmatrix},\ \ 
Q_i=\begin{pmatrix}
A_{i1} &\cdots &A_{ir} 
\end{pmatrix},\quad(i\neq k)
\\
&A_k+c=P_kQ_k;\quad P_k=I_n,\ \ Q_k=A_k+c. 
\end{split}
\end{equation} 
For the decomposition of $A-\rho$,  we use the following lemma. 

\begin{Lemma}\label{lem:Xxieta}
Let $X=(X_{ij})_{i,j=1}^{r}\in{\rm Mat}(n;\mathbb{C})$ 
be an $n\times n$ block matrix of type 
$(n_1,\ldots,n_r)$, $n_1+\cdots+n_r=n$, where 
$X_{ij}\in{\rm Mat}(n_i,n_j;\mathbb{C})$ $(i,j=1,\ldots,r)$.  
Suppose that 
${\rm rank}\, X_{kk}=n_k$ and ${\rm rank}\, X\ge n_k$ 
for an index $k=1,\ldots,r$ and set $l={\rm rank}\, X -n_k$.  
Then there exist matrices 
\begin{equation}
\xi_i\in{\rm Mat}(n_i,l;\mathbb{C})\quad(1\le i\le r,\ i\ne k),\quad
\eta_j\in{\rm Mat}(l,n_j;\mathbb{C})\quad(1\le j\le r,\ i\ne k)
\end{equation}
such that 
\begin{equation}\label{eq:Xxieta}
X_{ij}=X_{ik}X_{kk}^{-1}X_{kj}+\xi_i\eta_j\quad(1\le i,j\le r;\ i,j\ne k). 
\end{equation}
Furthermore, the 
$(n-n_k)\times l$ matrix $\xi=(\xi_i)_{1\le i\le r;\,i\ne k}$ and the 
$l\times (n-n_k)$ matrix $\eta=(\eta_j)_{1\le j\le r;\,i\ne k}$ are of maximal rank. 
\end{Lemma}
\begin{proof}
Without loss of generality we assume $k=r$ and set 
\begin{equation}
A=(X_{ij})_{i,j=1}^{r-1},\quad B=(X_{ir})_{i}^{r-1}, 
C=(X_{rj})_{j=1}^{r-1},\quad D=X_{rr}. 
\end{equation}
Noting that $D$ is invertible, we decompose this matrix into the form 
\begin{equation}
X=
\begin{pmatrix}
A & B\\
C & D
\end{pmatrix}
=
\begin{pmatrix}
1 & BD^{-1}\\
0 & 1
\end{pmatrix}
\begin{pmatrix}
A-BD^{-1}C& 0\\
0 & D
\end{pmatrix}
\begin{pmatrix}
1& 0\\
D^{-1}C & 1
\end{pmatrix}
\end{equation}
Since ${\rm rank}(A-BD^{-1}C)={\rm rank}\ X-n_r=l$, 
there exist matrices $\xi\in{\rm Mat}(n-n_r,l;\mathbb{C})$ 
and 
$\eta\in{\rm Mat}(l,n-n_r;\mathbb{C})$ of rank $l$ such that 
$A-BD^{-1}C=\xi\eta$.   
From $A=BD^{-1}C+\xi\eta$ we obtain
\begin{equation}
X_{ij}=X_{ir}X_{rr}^{-1}X_{rj}+\xi_i\eta_j\qquad(i,j=1,\ldots,r-1)
\end{equation}
where $\xi=(\xi_i)_{i=1}^{r-1}$ and $\eta=(\eta_j)_{j=1}^{r-1}$. 
\end{proof}

If we apply Lemma \ref{lem:Xxieta} to the matrix $A-\rho$ 
with the invertible block $A_{kk}-\rho$, 
then 
formula \eqref{eq:Xxieta} implies the decomposition 
$A-\rho=P_0Q_0$ of $A-\rho$ where 
\begin{equation}
P_0=
\begin{pmatrix}
A_{1k}(A_{kk}-\rho)^{-1}& \xi_{1}\\
\vdots&\vdots\\
A_{k-1,k}(A_{kk}-\rho)^{-1} & \xi_{k-1}\\
 1 & 0\\
A_{k+1,k}(A_{kk}-\rho)^{-1}& \xi_{k+1}\\
\vdots&\vdots\\
A_{rk}(A_{kk}-\rho)^{-1} & \xi_{r}
\end{pmatrix},
\quad
Q_0=\begin{pmatrix}
A_{k1}&\ldots &A_{k,k-1}&A_{kk}-\rho &
A_{k,k+1}&\ldots & A_{kr}\\
\eta_1 &\ldots&\eta_{k-1}& 0 &\eta_{k+1}&\ldots &\eta_r
\end{pmatrix}.
\end{equation}
In this case, 
the linear mapping 
\begin{equation}
\mathbb{C}^{nr}\to\mathbb{C}^{nr}:
\quad
(v_1,\ldots,v_r)^{}\mapsto
(v_1-v_k,\ldots,v_{k-1}-v_k,v_k,v_{k+1}-v_k,\ldots,v_r-v_k)
\end{equation}
induces the isomorphism from 
the quotient space 
$\mathbb{C}^{nr}/\big((\bigoplus_{p\neq k}{\rm Ker}A_p\oplus{\rm Ker}(A_k+c))
\oplus{\rm Ker}B\big)$ 
for the middle convolution to 
\begin{equation}
(\mathbb{C}^n/{\rm Ker}A_1)\oplus\cdots\oplus (\mathbb{C}^n/{\rm Ker}A_{k-1})
\oplus(\mathbb{C}^n/{\rm Ker}(A-\rho))\oplus
(\mathbb{C}^n/{\rm Ker}A_{k+1})\oplus\cdots\oplus 
(\mathbb{C}^n/{\rm Ker}A_{r}).
\end{equation}
Therefore we redefine $Q$ as
\begin{equation}
\arraycolsep=2pt 
Q=\begin{pmatrix}
\,Q_1 &&-Q_1\\[-4pt]
& \ddots&\vdots \\[-4pt]
& & Q_0 \\[-4pt]
& & \vdots& \ddots\\[-4pt]
&&-Q_r&& Q_r\,
\end{pmatrix}. 
\end{equation}
Then,  from the convolution 
\begin{equation}\label{eq:BZ}
(x-T)\frac{d\ }{dx}Z=BZ;\quad
B=(B_{ij})_{i,j=1}^{r},\quad 
B_{ij}=A_j+\delta_{jk}c-\delta_{ij}(\rho+c), 
\end{equation}
by the transformation $\widehat{Z}=QZ$
we obtain the Schlesinger system 
\begin{equation}\label{mcSch}
\frac{d}{dx}\widehat Z=
\sum_{l=1}^r\frac{\widehat B_l+\rho\delta_{kl}}{x-t_l}\widehat Z, 
\qquad \widehat{B}_lQ = QB_l\ \ (l=1,\ldots,r)
\end{equation}
corresponding to the operation \eqref{mcadd}, 
where 
\begin{equation}
\arraycolsep=2pt
\widehat B_l=
\begin{pmatrix}
 & &{\large O}& &  \\[4pt]
A_{l1} & \cdots& Q_lP_0 & \cdots & A_{lr} \\[4pt]
 & &{\large O}& & \\
\end{pmatrix}
-(\rho +c)E_{ll}
\ \  (l\neq k),\quad
\widehat B_k=
\begin{pmatrix}
-Q_1 \\[-3pt]
\vdots \\[-2pt]
Q_0 \\[-3pt]
\vdots \\[-2pt]
-Q_r
\end{pmatrix}
\begin{pmatrix}
P_1 & \cdots P_0 &\cdots & P_r
\end{pmatrix}
\end{equation}
with $Q_lP_0$ in the $(l,k)$ block of $\widehat{B}_l$. 

Then the Schlesinger system \eqref{mcSch} is transformed into the Okubo system
\begin{equation}\label{eq:mcadd}
(x-T)\frac{d}{dx}W=A^{mc}W,\qquad
A^{mc}={\rm Ad}(G)(\widehat B_{1}+\cdots +\widehat B_r+\rho)
\end{equation}
for $W=G\widehat{Z}$, 
where 
\begin{equation}
A^{mc}=\begin{pmatrix}
& A_{1k}(A_{kk}+c)(A_{kk}-\rho)^{-1} &(\rho+c)\xi_{1} &&\\[-3pt]
\ A_{ij}-(\rho+c)\delta_{ij}\ & \vdots & \vdots & A_{ij} \\[-3pt]
& A_{k-1,k}(A_{kk}+c)(A_{kk}-\rho)^{-1} &(\rho+c)\xi_{k-1} &&\\[4pt]
A_{k1}\ \ \ldots\ \ A_{k,k-1}  & A_{kk} & 0 & A_{k,k+1}\ \ \ldots \ \ A_{kr} \\[2pt]
\eta_{1}\ \ \ \ldots\ \ \ \eta_{k-1}  
& 0 & \rho & \eta_{k+1}\ \ \ \ldots\ \ \ \eta_{r} \\[2pt]
& A_{k+1,k}(A_{kk}+c)(A_{kk}-\rho)^{-1} & (\rho+c)\xi_{k+1}& \\[-4pt]
A_{ij} & \vdots & \vdots & A_{ij}-(\rho+c)\delta_{ij}\\[-3pt]
& A_{rk}(A_{kk}+c)(A_{kk}-\rho)^{-1} & (\rho+c)\xi_{r}& 
\end{pmatrix}
\end{equation}
and 
\begin{equation}
G=
\begin{pmatrix}
1 &&& A_{1k}(A_{kk}-\rho)^{-1}& \xi_1\\[-4pt]
 &\ddots &&\vdots&\vdots\\[-4pt]
&&1& A_{k-1k}(A_{kk}-\rho)^{-1}& \xi_{k-1}\\[2pt]
 &&&1 & 0\\
 &&&0 & 1\\
&&& A_{k+1k}(A_{kk}-\rho)^{-1}& \xi_{k+1}&1\\[-4pt]
 &&&\vdots&\vdots& &\ddots\\[-4pt]
 &&& A_{rk}(A_{kk}-\rho)^{-1}& \xi_{r}&&&1
\end{pmatrix}. 
\end{equation}

We remark that the transformation from the convolution \eqref{eq:BZ} 
to the Okubo system \eqref{eq:mcadd} is 
given by $W=GQZ$: 
\begin{equation}
GQ=\begin{pmatrix}
Q_1 &  & &(-\rho I_{n_1},\,0) & & & \\
 &\ddots & & \vdots& & & \\
 & &Q_{k-1}&(0,\,-\rho I_{n_{k-1}},\,0)   & & & \\
 & & & Q_0& &  &  \\
 & & & (0,\,-\rho I_{n_{k+1}},\,0) &Q_{k+1}&  & \\
 & & & \vdots && \ddots & \\
 & & &(0,\,-\rho I_{n_r})& &  & Q_r
\end{pmatrix}. 
\end{equation} 

\begin{Lemma}\label{lem:mcadd}
If the Okubo system \eqref{eq:Okubo} 
satisfies the conditions ${\rm Ker}(A_k+c)=0$ and ${\rm Ker}(A_{kk}-\rho)=0$ 
for an index $k=1,\ldots,r$, then  
the Schlesinger system with residue matrices \eqref{mcadd} is 
equivalent to the Okubo system \eqref{eq:mcadd}.
\end{Lemma}

\subsection{Middle convolution for  monodromy matrices of Okubo type}
Let $\Psi(x)$ be the canonical solution matrix of the Okubo system
\eqref{eq:Okubo}. The analytic continuation $\Psi(x)$ by $\gamma_k$ 
$(k=1,\ldots,r)$ gives an $r$-tuple of monodromy matrices 
$\bs{M}$ as defined in Section $1.2$.  
Then $\bs{M}$  is {\em of Okubo type} in the sense that 
$M_i$ are expressed as
\begin{equation}
M_i=\begin{pmatrix}
1 & & & & \\
 &\ddots & & & \\
M_{i1} &\cdots &M_{ii} & \cdots &M_{ir} \\
 & & &\ddots & \\
 & & & & 1
\end{pmatrix} 
\qquad(i=1,\ldots,r)
\end{equation}
in terms of matices $M_{ij}\in {\rm Mat}(n_i,n_j;\mathbb{C})$ 
($1\le i,j\le r$).
We show that the operation 
\begin{equation}\label{MCadd}
{\rm Add}_{(1,\ldots,\lambda^{-1},\ldots,1)}\circ {\rm MC}_{\lambda s^{-1}}
\circ {\rm Add}_{(1,\ldots,s,\ldots,1)}(\bs{M})
\end{equation}
gives rises to monodromy matrices of Okubo type 
as in the case of Okubo system 
\eqref{eq:Okubo}.  
We decompose $M_i-1,\,(i=1,\ldots,r)$ and take $\cS_i$ as
\begin{equation}
M_i-1=\cP_i\cQ_i=
\begin{pmatrix}
0\\
I_{n_i}\\
0
\end{pmatrix}
\begin{pmatrix}
M_{i1} &\cdots M_{ii}-1 &\cdots & M_{ir}
\end{pmatrix},\quad 
\cS_{i}=
\begin{pmatrix}
0\\
(M_{ii}-1)^{-1}\\
0
\end{pmatrix}.
\end{equation}
Also, setting $M_0=\lambda M_1\cdots M_r$, 
we take the decomposition 
\begin{equation}
\begin{split}
M_0^{(k)}-1 
&=\cP_0^{(k)}\cQ_0^{(k)}
\\
&=\begin{pmatrix}
M_{1k}^{(k)}(M_{kk}^{(k)}-1)^{-1} & \xi_1 \\
\vdots &\vdots \\
M_{k-1,k}^{(k)}(M_{kk}^{(k)}-1)^{-1} & \xi_{k-1} \\
1 & 0\\
M_{k+1,k}^{(k)}(M_{kk}^{(k)}-1)^{-1} & \xi_{k+1} \\
\vdots &\vdots \\
M_{rk}^{(k)}(M_{kk}^{(k)}-1)^{-1} & \xi_{r} 
\end{pmatrix}
\begin{pmatrix}
M_{k1}^{(k)} &\cdots &M_{k,k-1}^{(k)} & M_{kk}^{(k)}-1 &
M_{k,k+1}^{(k)} &\cdots &M_{k,r}^{(k)} \\
\eta_1 &\cdots&\eta_{k-1}& 0&\eta_{k+1}&\cdots &\eta_r
\end{pmatrix}
\end{split}
\end{equation} 
for
\begin{equation}
M_0^{(k)}={\rm Ad}(M_{k+1}\cdots M_r)(M_0)
=M_{k+1}\cdots M_r M_1\cdots M_{k}
=(M^{(k)}_{ij})_{i,j=1}^{r}
\end{equation}
as in Lemma \ref{lem:Xxieta}, 
and define 
$\cP_0=( M_{k+1}\cdots M_r)^{-1}\cP_0^{(k)}$, 
$\cQ_0=\cQ_0^{(k)}(M_{k+1}\cdots M_r)$,
$\cS_0$
so that $M_0-1=\cP_0\cQ_0$, $\cQ_0\cS_0=I_{n_0}$. 
As in the case of the differential system, 
the linear mapping 
$\mathbb{C}^{nr}\to\mathbb{C}^{nr}:$
\begin{equation}
\small
(v_1,\ldots,v_r)^{}\mapsto
(v_1-sM_2\cdots v_k,\ldots,v_{k-1}-sM_kv_k,v_k,v_{k+1}-M_{k+1}^{-1}v_k,
\ldots,v_r-(M_{k+1}\cdots M_r)^{-1}v_k)
\end{equation}
induces the isomorphism from 
the quotient space 
$\mathbb{C}^{nr}/\big((\bigoplus_{p\neq k}{\rm Ker}(M_{p}-1)
\oplus{\rm Ker}(sM_k-1))
\oplus{\rm Ker}(N_0-1)\big)$ 
for the middle convolution to 
\begin{equation}
\begin{split}
&(\mathbb{C}^n/{\rm Ker}(M_1-1))
\oplus\cdots\oplus (\mathbb{C}^n/{\rm Ker}(M_{k-1}-1))
\\
&
\oplus(\mathbb{C}^n/{\rm Ker}(M_0^{(k)}-1))
\oplus
(\mathbb{C}^n/{\rm Ker}(M_{k+1}-1))\oplus\cdots\oplus 
(\mathbb{C}^n/{\rm Ker}(M_r-1)).
\end{split}
\end{equation}
Therefore, taking block matrices
\begin{equation}
{\footnotesize
\arraycolsep=1pt
\cQ=\begin{pmatrix}
\cQ_1 & & & -s\cQ_1 M_{2}\cdots M_{k} & & &\\[-4pt]
 &\ddots && \vdots & & &\\[-4pt]
 & &\cQ_{k-1} & s\cQ_{k-1} M_k & &  &\\
 & & &\cQ_0(M_{k+1}\cdots M_r)^{-1} & & &\\
 & & & \cQ_{k+1}M_{k+1}^{-1} &\cQ_{k+1} & & \\[-4pt]
 & & & \vdots & &\ddots & \\[-4pt]
 & & & \cQ_r(M_{k+1}\cdots M_r)^{-1} & & &\cQ_r
\end{pmatrix}, 
\quad
\cS=
\begin{pmatrix}
\cS_1 & && sM_{2}\cdots M_r\cS_0 & & \\[-4pt]
 &\ddots &&\vdots & & \\[-4pt]
&&\cS_{k-1}&s M_k\cdots M_r\cS_0&\\
 & &&M_{k+1}\cdots M_r\cS_0 & & \\
 & &&M_{k+2}\cdots M_r\cS_0 & \cS_{k+1}& \\[-4pt]
 & && \vdots && \ddots & \\[-4pt]
 & && \cS_0 & && \cS_r
\end{pmatrix}},
\end{equation}
we obtain the monodromy matrices 
$(\widehat{N}_1,\ldots ,\widehat{N}_r)=(\cQ N_1\cS,\ldots,\cQ N_r \cS)$
corresponding to 
${\rm MC}_{\lambda s^{-1}}\circ {\rm Add}_{(1,\ldots,s,\ldots,1)}(\bs{M})$ as
\begin{equation}
\arraycolsep=1pt
\widehat N_{i}=
\begin{pmatrix}
1 & & & & & &\\[-4pt]
 &  \ddots & & & &\\
\lambda s^{-1}M_{i1} &\cdots  
&\lambda s^{-1}M_{ii} & \cdots &M_{ir}\\[-4pt]
&&&\ddots && \\
&&&&1
\end{pmatrix}\quad(i\ne k)
\end{equation}
with the $(i,k)$-block replaced by $\cQ_i\cP_0$, and 
\begin{equation}
\widehat N_k-1=
\begin{pmatrix}
-s\cQ_1M_2\cdots M_k \\
\vdots \\
-s\cQ_{k-1}M_k\\
\cQ_0(M_{k+1}\cdots M_r)^{-1}\\
-\cQ_{k+1}M_{k+1}^{-1}\\
\vdots \\
-\cQ_{r}(M_{k+1}\cdots M_r)^{-1}
\end{pmatrix}
\begin{pmatrix}
\lambda s^{-1}\cP_1 &\cdots &\cP_0 &\cdots &\cP_r
\end{pmatrix}.
\end{equation}
We remark here 
that the $i$-th rows of $M_0^{(k)}$ and $(M_{0}^{(k)})^{-1}$ 
are given respectively as follows: 
\begin{equation}\label{M0k}
\begin{split}
(M_{0}^{(k)})_{i}&=\begin{cases}
(\lambda M_{i}\cdots M_rM_{1}\cdots M_{k})_{i}\quad 
&(k+1 \leq i\leq r) \\[2pt]
( \lambda M_{i}\cdots M_{k})_{i}\quad 
&(1\leq i\leq k),
\end{cases}\\[2pt]
(M_{0}^{(k)})^{-1}_i&=\begin{cases}
( \lambda^{-1}M_{i}^{-1}\cdots M_{k+1}^{-1})_{i}\quad 
&(k+1\leq i\leq r) \\[2pt]
( \lambda^{-1} M_{i}^{-1}\cdots M_{1}^{-1}M_{r}^{-1}\cdots
 M_{k+1}^{-1})_{i}\ \ 
&(1\leq i\leq k).
\end{cases}
\end{split}
\end{equation}
Using these relations \eqref{M0k}, we rewrite $\widehat N_k-1$ as 
\begin{equation}
\widehat N_k-1=
\begin{pmatrix}
-s\big(\lambda^{-1}M_{ij}^{(k)}-\delta_{ij}\big)
_{i=1,\ldots,k-1\atop j=1,\ldots,r\hfill}
\\
\cQ_0^{(k)}\\[2pt]
-\big(\delta_{ij}-\lambda \widetilde{M}_{ij}^{(k)}\big)
_{i=k+1,\ldots,r\atop j=1\ldots r\hfill}
\end{pmatrix}
\begin{pmatrix}
\lambda s^{-1}\cP_1 &\cdots &\cP_0 &\cdots &\cP_r
\end{pmatrix},
\end{equation}
where 
$(M_0^{(k)})^{-1}=(\widetilde{M}^{(k)}_{ij})_{i,j=1}^r$. 
Noting that $M_{ij}^{(k)}-\delta_{ij}-
M_{ik}^{(k)}(M_{kk}^{(k)}-1)^{-1}M_{kj}^{(k)}=\xi_i\eta_j$,
we set 
\begin{equation}
\cX=\begin{pmatrix}
1 & & &
\lambda^{-1}s(\cP_0^{(k)})_{1}
 \\[-2pt]
 &\ddots & &
\vdots
 \\[-2pt]
 & &1 &
\lambda^{-1}s(\cP_0^{(k)})_{k-1}\\[2pt]
 & && 1 \\
&&
&\lambda
(M_0^{(k)})^{-1}_{k+1}
\cP_0^{(k)} & 1
\\[-2pt]
&&&
\vdots & &\ddots 
\\[-2pt]
&&
&\lambda
(M_0^{(k)})^{-1}_{r}
\cP_0^{(k)} & & &1
\end{pmatrix},
\end{equation}
where $(\cP_0^{(k)})_i$ stands for the $i$-th row of $\cP_0^{(k)}$. 
Using this matrix $\cX$ 
for the conjugation, from 
$\widehat{\bs{N}}=(\widehat{N}_1,\ldots,\widehat{N}_r)$ 
we obtain the following monodromy matrices of Okubo type 
$\bs{M}^{mc}=(M_1^{mc},\ldots,M_r^{mc})
={\rm Add}_{(1,\ldots,\lambda^{-1},\ldots,1)}({\rm Ad}(\cX)\widehat{\bs{N}})$:
\begin{equation}\label{Mon:MCadd}
\begin{split}
M_{i}^{mc}&=
\begin{pmatrix}
1&&&&&&\\[-4pt]
&&\ddots&&&& \\
\lambda s^{-1} M_{i1}& 
& \cdots &\lambda s^{-1}M_{ii}&\cdots &
\begin{pmatrix}
M_{i(k1)}^{mc} & M_{i(k2)}^{mc}
\end{pmatrix} 
\cdots &M_{ir}\\
&&&&\ddots&&\\[-4pt]
&&&&&&1
\end{pmatrix}\,(i<k)\\
M_{i}^{mc}&=
\begin{pmatrix}
1&&&&&&\\[-4pt]
&&\ddots&&&& \\
\lambda s^{-1}M_{i1} 
\cdots &
\begin{pmatrix}
M_{i(k1)}^{mc} & M_{i(k2)}^{mc}
\end{pmatrix} 
& \cdots 
 &\lambda s^{-1}M_{ii}&
 & \cdots&M_{ir}\\
&&&&&\ddots&\\[-4pt]
&&&&&&1
\end{pmatrix}\,(k<i)\\
M_k^{mc}&=
\begin{pmatrix}
1 &&&&&\\
&\ddots&&&&\\
\lambda s^{-1}M_{k1} & \cdots & M_{kk} & 0 & \cdots & M_{kr}\\
s^{-1}\eta_{1} & \cdots &0 &\lambda^{-1} & \cdots &\lambda^{-1}\eta_{r}\\
&&&&\ddots&\\
&&&&&1
\end{pmatrix}
\end{split}
\end{equation}
where $M_{i(k1)}^{mc}$ and $M_{i(k2)}^{mc}$ 
are computed as follows: 
\begin{equation}\label{*12}
\begin{split}
M_{i(k1)}^{mc}
&=M_{ik}+(1-\lambda^{-1}s)(\sum_{j=i+1}^{k-1}
M_{ij}M^{(k)}_{jr}(M^{(k)}_{kk}-1)^{-1}-M^{(k)}_{ir}(M^{(k)}_{kk}-1)^{-1})\\
&=M_{ik}-(1-\lambda^{-1}s )M_{ik}M^{(k)}_{kk}(M^{(k)}_{kk}-1)^{-1}\\
&=\lambda^{-1} sM_{ik}(M^{(k)}_{kk}-\lambda s^{-1})(M^{(k)}_{kk}-1)^{-1},\\
M_{i(k2)}^{mc}
&=(1-\lambda^{-1}s )
(\sum_{j=i+1}^{k-1}M_{ij}\xi_{j}-\xi_{i})\\
\end{split}
\end{equation}
for $i<k$, and 
\begin{equation}
\begin{split}
M_{i(k1)}^{mc}
&=M_{ik}-\lambda(1-\lambda s^{-1})(\sum_{j=k+1}^i
M_{ij}\widetilde M^{(k)}_{jr}(M^{(k)}_{kk}-1)^{-1})\\
&=M_{ik}+\lambda(1-\lambda s^{-1})M_{ik}(M^{(k)}_{kk}-1)^{-1}\\
&=M_{ik}(M^{(k)}_{kk}-\lambda s^{-1})(M^{(k)}_{kk}-1)^{-1},
\\
M_{i(k2)}^{mc}
&=\lambda(1-\lambda s^{-1})\sum_{j=k+1}^{i}M_{ij}
(\sum_{p=1,p\neq k}^{r}\widetilde M^{(k)}_{jp}\xi_{p}).
\end{split}
\end{equation}
for $i>k$.  
Note that 
$\bs{M}^{mc}$ is the $r$-tuple of monodromy matrices for the 
fundamental solution matrix 
$Y^{mc}(x)=GQ(x-t_k)^{\rho}Z(x)\cS\cX^{-1}$, 
where $Z(x)=I^{-\rho-c}((x-t_k)^c \Psi(x))$. 
Therefore we need to specify the right inverse
$\cS\cX^{-1}$ of $\cX\cQ$ to solve the connection problem.
 In this case the matrices $\cX\cQ$ and
$\cS\cX^{-1}$ are given as
\begin{equation}
\cX\cQ=\begin{pmatrix}
(\cQ_i\delta_{ij})_{i,j=1}^{k-1}&
(s(1-\lambda^{-1})I_{n_i}\delta_{ij})
_{i=1,\ldots,k-1\atop j=1,\ldots,r\hfill}
&0\\
0 &Q_0^{(k)} &0\\
0 & ((\lambda-1)\delta_{ij})_{i=k+1,\ldots,r\atop j=1,\ldots,r\hfill}
&(\cQ_i\delta_{ij})_{i,j=k+1}^r
\end{pmatrix},
\end{equation}
\begin{equation}
\cS\cX^{-1}=\begin{pmatrix}
\cS_1&&&(s(\lambda^{-1}-1)\cS_1,0)\cS_{0}^{(k)}&&&\\
&\ddots&&\vdots&&&\\
&&\cS_{k-1}&(0,s(\lambda^{-1}-1)\cS_{k-1},0)\cS_{0}^{(k)}&&&\\
&&&\cS_{0}^{(k)}&&&\\
&&&(0,(1-\lambda)\cS_{k+1},0)\cS_{0}^{(k)}&\cS_{k+1}&&\\
&&&\vdots&&\ddots&\\
&&&(0,(1-\lambda)\cS_r)\cS_{0}^{(k)}&&&\cS_r
\end{pmatrix}. 
\end{equation}
We remark that one can take $\cS_0^{(k)}$ in the form 
\begin{equation}
\cS_0^{(k)}=\begin{pmatrix}
0 & \widetilde \eta_{1}\\
\vdots & \vdots \\
0 & \widetilde \eta_{k-1}\\
(M_{kk}^{(k)}-1)^{-1}& 0 \\
0 & \widetilde \eta_{k+1}\\
\vdots & \vdots \\
0 & \widetilde \eta_{r}
\end{pmatrix}. 
\end{equation}
Summarizing these arguments, we conclude: 
\begin{Proposition}\label{Prop:MCadd}
The operation \eqref{MCadd} gives rise to the $r$-tuple of monodromy matrices 
of Okubo type
$\bs{M}^{mc}$ as in \eqref{Mon:MCadd}, \eqref{*12}.
Furthermore $\bs{M}^{mc}$ is realized as the monodromy matrices 
for the fundamental solution matrix $Y^{mc}(x)=GQ(x-t_k)^{\rho}Z(x)\cS\cX^{-1}$, where 
$Z(x)=I^{-\rho-c}((x-t_k)^{c}\Psi(x))$. 
\end{Proposition}

\subsection{Recurrence relations 
for the connection coefficients}
We consider the 
Okubo system \eqref{eq:mcadd} 
obtained from 
the Okubo system \eqref{eq:Okubo} 
by the operation \eqref{mcadd}. 
Let 
$\Psi^{mc}(x)=(\Psi_1^{mc}(x),\ldots,\Psi_{(k1)}^{mc}(x),\Psi_{(k2)}^{mc}(x),
\ldots,\Psi_{r}^{mc}(x))$ be the canonical
solution matrix of \eqref{eq:mcadd}. Then the monodromy matrices 
$\Gamma_i^{mc}$ $(i=1,\ldots,r)$
corresponding to \eqref{eq:Mk} are represented as 
\begin{equation}
\begin{split}
\Gamma_i^{mc}&=
\begin{pmatrix}
1&&&&\\
&\ddots&&&\\
(e(A_{ii}-\rho-c)-1)C_{i1}^{mc}&\cdots & e(A_{ii}-\rho-c)&
 \cdots & (e(A_{ii}-\rho-c)-1)C_{ir}^{mc}\\
&&&\ddots& \\
&&&&1 
\end{pmatrix}
\quad(i\ne k)
\end{split}
\end{equation}
where $C_{ik}^{mc}=(C_{i(k1)}^{mc},C_{i(k2)}^{mc})$, 
and 
\begin{equation}
\begin{split}
\Gamma_k^{mc}&=
\begin{pmatrix}
1&&&&&\\
&\ddots&&&&\\
(e(A_{kk})-1)C_{(k1)1}^{mc}&\cdots&e(A_{kk})&0&\cdots&(e(A_{kk})-1)C_{(k1)r}^{mc}\\
(e(\rho)-1)C_{(k2)1}^{mc}&\cdots &0&e(\rho)&\cdots
 &(e(\rho)-1)C_{(k2)r}^{mc}\\
&&&&\ddots&\\
&&&&&1
\end{pmatrix}.
\end{split}
\end{equation}
We now investigate the relation between $C_{ij}$ and 
$C_{ij}^{mc},C_{(k1)j}^{mc}$. 
From the forms of \eqref{Mon:MCadd}, the fundamental solution matrix
$Y^{mc}(x)=GQ(x-t_k)^{\rho}Z(x)\cS\cX^{-1}$ is transformed into 
$\Psi^{mc}(x)=Y^{mc}(x)R^{-1}$ 
by some block diagonal matrix 
$R=(R_1,\ldots,R_{(k1)},R_{(k2)},\ldots,R_r)$. 
Therefore we have recurrence relations
\begin{equation}
\begin{split}
C_{ij}^{mc}&=
\begin{cases}
(e(A_{ii}-\rho-c)-1)^{-1}R_i(e(A_{ii})-1)\lambda
 s^{-1}C_{ij}R_j^{-1}\quad &(j<i;\,i,j\neq k)\\
(e(A_{ii}-\rho-c)-1)^{-1}R_i(e(A_{ii})-1)C_{ij}R_j^{-1}\quad
&(i<j;\,i,j\neq k)
\end{cases}\\
C_{i(k1)}^{mc}&=
\begin{cases}
(e(A_{ii}\!-\!\rho\!-\!c)\!-\!1)^{-1}R_{i}(e(A_{ii})\!-\!1)
C_{ik}(e(A_{kk}\!+\!c)\!-\!1)
(e(A_{kk}\!-\!\rho)\!-\!1)^{-1}R_{(k1)}^{-1}
\ &(i<k)\\
(e(A_{ii}\!-\!\rho\!-\!c)\!-\!1)^{-1}R_{i}
(e(A_{ii})\!-\!1)\lambda s^{-1} 
C_{ik}(e(A_{kk}\!+\!c)\!-\!1)
(e(A_{kk}\!-\!\rho)\!-\!1)R_{(k1)}^{-1}
\ &(k<i)
\end{cases}\\
C_{(k1)j}^{mc}&=
\begin{cases}
R_{(k1)}\lambda s^{-1}C_{kj}R_j^{-1}\quad &(j<k)\\
R_{(k1)} C_{kj}R_j^{-1}\quad &(k<j).  
\end{cases}
\end{split}
\end{equation}
The new connection coefficients $C_{(k2)j},C_{i(k2)}$ are 
also given as
\begin{equation}
\begin{split}
C_{(k2)j}^{mc}&=\begin{cases}
R_{(k2)}s^{-1}\eta_jR_j^{-1}\quad &(j<k)\\
R_{(k2)}\lambda^{-1}\eta_jR_j^{-1}&(k<j)
\end{cases}\\
C_{i(k2)}^{mc}&=
\begin{cases}
(e(A_{ii}\!-\!\rho\!-\!c)\!-\!1)^{-1}
R_i(1\!-\!e(\rho\!+\!c))(\sum_{j=i+1}^{k-1}C_{ij}\xi_j-\xi_i)R_{(k2)}^{-1}
\ &(i<k)\\
(e(A_{ii}\!-\!\rho\!-\!c)\!-\!1)^{-1}
R_{i}e(-\rho)(1\!-\!e(-\rho\!-\!c))(\sum_{j=k+1}^rC_{ij}\sum_{p=1,p\neq k}^{r}
\widetilde C_{jp}^{(k)}\xi_p)R_{(k2)}^{-1}\ &(k<i). 
\end{cases}
\end{split}
\end{equation}
Looking at the diagonal blocks of the 
solution matrix $Z(x)$ and 
$GQ,\cS\cX^{-1}$,  we can compute $R_i$ as follows:
\begin{equation}
\begin{split}
R_i&=A_{ii}\lim_{x\to t_i}(x-t_i)^{-A_{ii}+\rho+c}(x-t_k)^{\rho}
\int_{L_i}(x-u)^{-\rho-c}(u-t_k)^c\Psi_i(u)\frac{du}{u-t_i}(e(A_{ii})-1)^{-1}\\
&=A_{ii}(t_i-t_k)^{\rho+c}\widetilde B(A_{ii},-\rho-c+1)(e(A_{ii})-1)^{-1}\\
R_{(k1)}&=(A_{kk}-\rho)\lim_{x\to t_k}(x-t_k)^{-A_{kk}+\rho}
\int_{L_i}(x-u)^{-\rho-c}(u-t_k)^c\Psi_k(u)\frac{du}{u-t_k}
(e(A_{kk}-\rho)-1)^{-1}\\
&=(A_{kk}-\rho)\widetilde B(A_{kk}+c,-\rho-c+1)
(e(A_{kk}-\rho)-1)^{-1}
\end{split}.
\end{equation}
Using these formulas, the recurrence relations for the connection coefficients 
are rewritten as 
\begin{equation}\label{cij}
\begin{split}
C_{ij}^{mc}&=
(e(A_{ii}-\rho-c)-1)^{-1}A_{ii}(t_i-t_k)^{\rho+c}
\widetilde B(A_{ii},-\rho-c+1)(e(-\rho-c))^{\delta(i<j)}
C_{ij}\\
&\quad\cdot (e(A_{jj})-1)(\widetilde B(A_{jj},-\rho-c+1))^{-1}
(t_j-t_k)^{-c-\rho}A_{jj}^{-1}\\
&=
\begin{cases}
\displaystyle
\left(\frac{t_i-t_k}{t_j-t_k}\right)^{\rho+c}
\frac{e(\frac{-1}{2}(\rho+c))\Gamma(\rho+c-A_{ii})}{\Gamma(-A_{ii})}
C_{ij}\frac{\Gamma(A_{jj}-\rho-c+1)}{\Gamma(A_{jj}+1)}
\ &(j<i;\,i,j\neq k)\\[10pt]
\displaystyle
\left(\frac{t_i-t_k}{t_j-t_k}\right)^{\rho+c}
\frac{e(\frac{1}{2}(\rho+c))\Gamma(\rho+c-A_{ii})}{\Gamma(-A_{ii})}
C_{ij}\frac{\Gamma(A_{jj}-\rho-c+1)}{\Gamma(A_{jj}+1)}
\ &(i<j;\,i,j\neq k),
\end{cases}
\end{split}
\end{equation}
where $\delta(i<j)=1$ if $i<j$ and $\delta(i<j)=0$ otherwise. 
Similarly, 
\begin{equation}\label{Ci(k1)}
\begin{split}
C_{i(k1)}^{mc}&=
\begin{cases}
\displaystyle
(t_i-t_k)^{\rho+c}e(\tfrac{1}{2}(\rho+c))
\frac{\Gamma(\rho+c-A_{ii})}{\Gamma(-A_{ii})}
C_{ik}\frac{\Gamma(A_{kk}-\rho)}{\Gamma(A_{kk}+c)}
\quad &(i<k)\\[10pt]
\displaystyle
(t_i-t_k)^{\rho+c}e(\tfrac{-1}{2}(\rho+c))
\frac{\Gamma(\rho+c-A_{ii})}{\Gamma(-A_{ii})}
C_{ik}\frac{\Gamma(A_{kk}-\rho)}{\Gamma(A_{kk}+c)}
\quad &(k<i)
\end{cases}
\end{split}
\end{equation}
\begin{equation}\label{C(k1)j}
\begin{split}
C_{(k1)j}^{mc}
&=
\begin{cases}
\displaystyle
-e(\tfrac{-1}{2}(\rho+c))(t_j-t_k)^{-\rho-c}
\frac{\Gamma(1+\rho-A_{kk})}{\Gamma(1-A_{kk}-c)}C_{kj}
\frac{\Gamma(A_{jj}-\rho-c+1)}{\Gamma(A_{jj}+1)}
\quad &(j<k)\\[10pt]
\displaystyle
-e(\tfrac{1}{2}(\rho+c))(t_j-t_k)^{-\rho-c}
\frac{\Gamma(1+\rho-A_{kk})}{\Gamma(1-A_{kk}-c)}C_{kj}
\frac{\Gamma(A_{jj}-\rho-c+1)}{\Gamma(A_{jj}+1)}
\quad &(k<j)
\end{cases}
\end{split}
\end{equation}
\begin{Theorem}\label{connection}
The connection coefficients $C_{ij}^{mc}$
and $C_{i(k1)}^{mc}$, $C_{(k1)j}^{mc}$ 
$(i,j\ne k)$ 
for the canonical 
solution matrix $\Psi^{mc}(x)$ are 
determined from the connection coefficients $C_{ij}$ 
for $\Psi(x)$ by the recurrence formulas 
\eqref{cij} and \eqref{Ci(k1)}, \eqref{C(k1)j}
respectively. 
\end{Theorem}

It seems that the matrices 
$R_{(k2)}$ and $C^{mc}_{(k2)}$ 
cannot be written 
as products of gamma functions in general, 
while they can be described as 
the limits 
as $x\to t_k$ of certain explicit integrals. 
In this paper we do not go into the detail of the 
description of $R_{(k2)}$ and $C^{mc}_{(k2)}$, 
since they can be determined by symmetry arguments 
in the context of the Okubo systems in Yokoyama's list
that will be discussed below. 

\section{Connection coefficients 
for the Okubo systems of types ${\rm I}$,\, ${\rm I}^*$,\,
 ${\rm II}$ and ${\rm III}$}

\subsection{Construction of canonical forms of the 
Okubo systems of types ${\rm I}$,\, ${\rm I}^*$,\, ${\rm II}$ and ${\rm III}$}

In what follows, we use the symbols 
 $({\rm I})_n$,\, $({\rm I}^*)_n$,\, $({\rm II})_{2n}$, and $({\rm III})_{2n+1}$
to refer to the corresponding tuples $\bs{A}$ of residue matrices 
as in \eqref{eq:Sch}.
As for the nontrivial eigenvalues, we also use the notations 
$\alpha^{(l)}_i=\alpha_i$, 
$\beta^{(l)}_i=\beta_i$, 
$\rho^{(l)}_i=\rho_i$
in order to specify the rank $l=n,n,2n,2n+1$ of the Okubo system.
\par\medskip
\noindent{\bf Case ${\rm I}$}: 
We first construct $({\rm I})_2=({\rm II})_2$ 
by ${\rm mc}_{\mu_1}(\alpha_1^{(1)},\beta_1^{(1)})$
starting from the differential equation 
\begin{equation}
\frac{dy}{dx}=\left(\frac{\alpha_1^{(1)}}{x-t_1}
+\frac{\beta_1^{(1)}}{x-t_2}\right)y
\end{equation}
of rank $1$.  The resulting system 
is given by
\begin{equation}\label{HGE}
(x-T)\frac{d}{dx}Y=AY;\quad 
A=\begin{pmatrix}
\alpha_1^{(1)}+\mu_1 & \beta_1^{(1)} \\
\alpha_1^{(1)} & \beta_1^{(1)}+\mu_1
\end{pmatrix}
=\begin{pmatrix}
\alpha_1^{(2)} & \beta_1^{(2)}-\rho_1^{(2)}\\
\alpha_1^{(2)}-\rho_1^{(2)} &\beta_1^{(2)}
\end{pmatrix}. 
\end{equation}
The canonical form of type $({\rm I})_2$ is obtained as the conjugation 
of $A$ by the matrix
\begin{equation}
G=\begin{pmatrix}
1& 0\\
0&\beta_1^{(2)}-\rho_1^{(2)}
\end{pmatrix}
\end{equation} 

Starting with $({\rm I})_2$, the Okubo systems 
${\rm (I)}_{n}$ can be constructed inductively 
by the Katz operations
\begin{equation}\label{eq:makeI}
\begin{split}
{\rm (I)}_{n+1}&={\rm add}_{(\rho,0)}\circ {\rm mc}_{-c-\rho} 
\circ {\rm add}_{(c,0)}{\rm (I)}_{n} \quad(n\geq 2). 
\end{split}
\end{equation}

\begin{Lemma}\label{lem:xietaEqI}
For the system \eqref{eq:CI} of type ${\rm (I)}_{n}$, 
$\alpha_n-\rho-K(\alpha-\rho)^{-1}L$ is a matrix of rank 1.  
It is expressed as 
\begin{equation}
\alpha_n-\rho_2-K(\alpha-\rho_2)^{-1}L=\xi\eta
\end{equation}
with $\xi$ and $\eta$ defined by 
\begin{equation}
\xi=-\frac{\prod_{k=1}^{n}(\rho-\rho_k)}{\prod_{k=1}^{n-1}(\rho-\alpha_k)},
\quad \eta=1
\end{equation}
\end{Lemma}
\noindent 

We show how the canonical form \eqref{eq:CI} of the Okubo system ${(\rm I)}_{n+1}$ 
is obtained from the canonical form \eqref{eq:CI} of $({\rm I})_{n}$ by 
the operation of \eqref{eq:makeI}.
By Lemma \ref{lem:mcadd} 
the Katz operation \eqref{eq:makeI} for $({\rm I})_{n}$ 
gives rise to 
the Okubo system $(x-T)\frac{d}{dx}W=A^{mc}W$ where
\begin{equation}
A^{mc}=\begin{pmatrix}
\alpha &0 &K \\
0&\rho&\eta \\
L(\alpha+c)(\alpha-\rho)^{-1} &(\rho+c)\xi &\alpha_n \\
\end{pmatrix}.  
\end{equation}
This matrix $A^{mc}$ is precisely 
the canonical form of Okubo system $({\rm I})_{n+1}$ 
with characteristic exponents 
$(\alpha_i^{(n+1)})_{i=1}^n,\,(\rho_i^{(n+1)})_{i=1}^n$
specified by 
\begin{equation}
\begin{cases}
&\alpha_i^{(n+1)}=\alpha_i \quad(1\leq i \leq n-1),\quad \alpha_n^{(n+1)}=\rho \\
&\rho_i^{(n+1)}=\rho_i, \quad \rho_n^{(n+1)}=-c.  
\end{cases}
\end{equation}

\par\medskip
\noindent{\bf Case ${\rm I}^*$}:
The system of type $({\rm I}^*)_n$ is obtained 
by ${\rm mc}_{\mu}$ with a generic parameter $\mu$ 
from the Schlesinger system
\begin{equation}\label{Sch(n,1)}
\frac{d}{dx}y=\sum_{k=1}^n\frac{\alpha_k}{x-t_k}y
\end{equation}
of rank 1, where $\alpha_k\in \mathbb{C}^*$ $(k=1,\ldots,n)$. 
The resulting system is given by 
\begin{equation}\label{eq:I*}
(x-T)\frac{d}{dx}Z=
\begin{pmatrix}
\alpha_1+\mu&\cdots&\alpha_n\\
\vdots&\ddots&\vdots \\
\alpha_1&\cdots&\alpha_n+\mu
\end{pmatrix}Z
\end{equation}
Then characteristic exponents $(\alpha_i^{(n)})_{i=1}^n$ and $\rho_1^{(n)},\rho_2^{(n)}$ are
determined as
\begin{equation}\label{eigenI*}
\alpha_i^{(n)}=\alpha_i+\mu,\quad \rho_1^{(n)}=\mu,\quad \rho_2^{(n)}=\sum_{k=1}^n\alpha_k+\mu.
\end{equation}

\par\medskip
\noindent{\bf Cases ${\rm II}$ and ${\rm III}$}:
The Okubo system ${\rm (III)}_3$ as in \eqref{eq:CIII} 
is obtained from the system \eqref{HGE}
by
${\rm add}_{(\rho,0)}\circ {\rm mc}_{-a_1-\rho} \circ {\rm add}_{(c,0)}$ 
with 
\begin{equation}
\begin{cases}
&\alpha^{(3)}_1=\alpha^{(2)}_1-a_1-\rho,\quad \alpha_2^{(3)}=\rho
\\
&\beta^{(3)}_1=\beta^{(2)}_1-\rho-a_1,\quad
\\
&\rho^{(3)}_1=-a_1,\quad
\rho^{(3)}_2=\rho_1^{(2)},\quad
\rho^{(3)}_3=\alpha^{(2)}_1+\beta^{(2)}_1-\rho_1^{(2)}. 
\end{cases} 
\end{equation}
 
The Okubo systems 
${\rm (II)}_{2n}$ and ${\rm (III)}_{2n+1}$ 
can be constructed inductively 
by the following Katz operations:
\begin{equation}\label{eq:makeII-III}
\begin{split}
{\rm (II)}_{2n}&={\rm add}_{(0,\rho_2^{(2n-1)})}\circ {\rm mc}_{-b_n-\rho_2^{(2n-1)}} 
\circ {\rm add}_{(0,b_{n-1})}{\rm (III)}_{2n-1}\quad (n\geq 2).\\
{\rm (III)}_{2n+1}&={\rm add}_{(\rho_2^{(2n)},0)}\circ {\rm mc}_{-a_n-\rho_2^{(2n)}} 
\circ {\rm add}_{(a_n,0)}{\rm (II)}_{2n}, \quad(n\geq 2)
\end{split}
\end{equation}

\begin{Lemma}\label{lem:xietaEq}
$(1)$ For the system \eqref{eq:CII} of type ${\rm (II)}_{2n}$, 
$\beta-\rho_2-K(\alpha-\rho_2)^{-1}L$ is a matrix of rank 1.  
It is expressed as 
\begin{equation}
\beta-\rho_2-K(\alpha-\rho_2)^{-1}L=\xi\eta
\end{equation}
with a column vector $\xi$ and a row vector $\eta$ defined by 
\begin{equation}
\xi_i=(\rho_2-\rho_1)\frac{\prod_{k\neq i\, 1\leq k\leq n}(\beta_k-\rho_1)}{\prod_{k=1}^n(\rho_2-\alpha_k)},
\quad
\eta_j=\frac{\prod_{k=1}^n(\beta_j+\alpha_k-\rho_1-\rho_2)}
{\prod_{k\neq j\, 1\leq k\leq n}(\beta_j-\beta_k)}.
\end{equation}
$(2)$ For the system \eqref{eq:CIII} of type ${\rm (III)}_{2n+1}$,
$\alpha-\rho_2-L(\beta-\rho_2)^{-1}K$ is a matrix of rank 1.  
It is expressed as 
\begin{equation}
\alpha-\rho_2-L(\beta-\rho_2)^{-1}K=\xi\eta
\end{equation}
with a column vector $\xi$ and a row vector $\eta$ 
defined by 
\begin{equation}
\xi_i=
\frac{\prod_{k\neq i,\, 1\leq k\leq n+1}(\alpha_k-\rho_1)}
{\prod_{k=1}^n(\rho_2-\beta_k)},
\quad
\eta_j=
(\alpha_j-\rho_2)
\frac{\prod_{ 1\leq k\leq n}(\beta_k+\alpha_{j}-\rho_1-\rho_2)}
{\prod_{k\neq j,\,1\leq k\leq n+1}(\alpha_j-\alpha_k)}. 
\end{equation}
\end{Lemma}
\noindent 
This lemma is essentially used in \cite{YHEq}
for the construction of canonical forms.  
The explicit forms of $\xi$ and $\eta$ 
described above can also be derived  
by the property of Cauchy matrices as in \cite{YHEq}.

\par\medskip
\noindent{\bf\boldmath From $(\rm III)_{2n-1}$ to $(\rm II)_{2n}$}
\par\smallskip
We show how the canonical form \eqref{eq:CII} of the Okubo system ${(\rm II)}_{2n}$ 
is obtained from the canonical form \eqref{eq:CIII} of $({\rm III})_{2n-1}$ by 
the operation of \eqref{eq:makeII-III}.
By Lemma \ref{lem:mcadd} 
the Katz operation \eqref{eq:makeII-III} for $({\rm III})_{2n-1}$ 
gives rise to 
the Okubo system $(x-T)\frac{d}{dx}W=A^{mc}W$ where
\begin{equation}
A^{mc}=\begin{pmatrix}
\alpha-\rho_2-b_n & K(\beta+b_{n-1})(\beta-\rho_2)^{-1} &(\rho_2+b_{n-1})\xi\\
L &\beta & 0 \\
\eta& 0 &\rho_2
\end{pmatrix}.  
\end{equation}
This matrix $A^{mc}$ is precisely 
the canonical form of Okubo system $({\rm II})_{2n}$ 
with characteristic exponents 
$(\alpha_i^{(2n)})_{i=1}^n$, 
$(\beta_i^{(2n)})_{i=1}^{n},\,(\rho_i^{(2n)})_{i=1}^3$
specified by 
\begin{equation}
\begin{cases}
&\alpha_i^{(2n)}=\alpha_i-\rho_2-b_{n-1} \quad(1\leq i \leq n),\\
&\beta_i^{(2n)}=\beta_i \quad(1\leq i\leq n),\quad
\beta_{n+1}^{(2n)}=\rho_2,\\
&\rho_1^{(2n)}=-b_{n-1}, \quad \rho_2^{(2n+1)}=\rho_1, \quad\rho_3^{(2n+1)}=\rho_3.  
\end{cases}
\end{equation}

\par\medskip
\noindent{\bf\boldmath 
From $(\rm II)_{2n}$ to $(\rm III)_{2n+1}$}
\par\smallskip
Similarly one can construct the canonical form 
\eqref{eq:CII} of the Okubo system 
$({\rm II})_{2n+2}$ from $({\rm III})_{2n+1}$ by the middle convolution.
From Lemma \ref{lem:mcadd}, the Schlesinger system of \eqref{eq:makeII-III} 
for $({\rm III})_{2n+1}$ is given by the Okubo system 
$(x-T)\frac{d}{dx}W=A^{mc}W$ 
of type $({\rm II}_{2n+2})$, 
where 
\begin{equation}
A^{mc}=
\begin{pmatrix}
\alpha & 0 & K  \\
0 &\rho_2 &\eta \\
L(\alpha+a_n)(\alpha-\rho_2)^{-1} &(\rho_2+a_n)\xi &\beta-a_n-\rho_2
\end{pmatrix}. 
\end{equation}
The 
characteristic exponents 
$(\alpha_i^{(2n+1)})_{i=1}^{n+1},\,(\beta_i^{(2n+1)})_{i=1}^{n},\,(\rho_i^{(2n+1)})_{i=1}^3$
are determined as 
\begin{equation}
\begin{cases}
\alpha_i^{(2n+1)}=\alpha_i\quad(1\leq i \leq n),\quad 
\alpha_{n+1}^{(2n+1)}=\rho_2,\\
\beta_i^{(2n+1)}=\beta_i-\rho_2-a_n\quad(1\leq i\leq n),\\
\rho_1^{(2n+1)}=-a_n, \quad \rho_2^{(2n+1)}=\rho_1, \quad \rho_3^{(2n+1)}=\rho_3.
\end{cases}
\end{equation}
Note that the passage 
from $({\rm II})_2$ to $({\rm III})_3$ is also included 
in this procedure with $\rho_2$ replaced by $\rho$.

\subsection{Recurrence relations 
for the connection coefficients of types ${\rm I}$, ${\rm I}^*$, ${\rm II}$ and 
${\rm III}$}
In what follows we prove our main theorems in Section $2$ using Recurrence relations 
for the connection coefficients. 
\par\medskip
\noindent{\bf Case ${\rm I}$}: 
We determine the connection coefficients 
$C_{i}=C_{i}^{(n)}$ and $D_{j}=D_{j}^{(n)}$ for the canonical solution 
matrix $\Psi^{(n)}(x)$ of 
the Okubo system $({\rm I})_n$ in the canonical form \eqref{eq:CI}.

Recall that 
the Okubo system of type $({\rm I})_n$ is constructed inductively 
by the Katz operation \eqref{eq:makeI}.  
Hence, by Proposition \ref{Prop:MCadd} 
the canonical solution matrix $\Psi^{(n)}(x)$
of type $({\rm I})_{n}$ is obtained inductively in the form 
\begin{equation}\label{sol:typeI}
\Psi^{(n+1)}(x)=GQ(x-t_1)^{\rho}
I^{-a_n-\rho}\big((x-t_1)^{a_n}\Psi^{(n)}(x)\big)\cS\cX R^{-1}, 
\end{equation}
where $R=(R_{(11)},R_{(12)},R_2)$. 
Also, by Theorem \ref{connection} we obtain the 
recurrence relations 
for the connection coefficients 
$C_{i}^{(n+1)}$ and $D_{i}^{(n+1)}$ $(i=1,\ldots,n-1)$ as follows: 
\begin{equation}
\begin{split}
C_{i}^{(n+1)}&=(C_{(11)2}^{mc})_{i}=
-e(\tfrac{1}{2}(\rho+a_n))(t_2-t_1)^{-\rho-a_n}
\frac{\Gamma(1+\rho-\alpha_i^{(n)})}{\Gamma(1-\alpha_i^{(n)}-a_n)}C_{i}^{(n)}
\frac{\Gamma(\alpha_n^{(n)}-\rho-a_n+1)}{\Gamma(\alpha_n^{(n)}+1)}\\
&=-e(\tfrac{1}{2}(\alpha_{n}^{(n+1)}-\rho_{n+1}^{(n+1)}))(t_2-t_1)^{\rho_{n+1}^{(n+1)}-\alpha_{n}^{(n+1)}}
\frac{\Gamma(1+\alpha_{n}^{(n+1)}-\alpha_i^{(n+1)})}{\Gamma(1+\rho_{n+1}^{(n+1)}-\alpha_i^{(n+1)})}
\frac{\Gamma(\alpha_{n+1}^{(n+1)}+1)}{\Gamma(\alpha_n^{(n)}+1)}C_{i}^{(n)}\\
D_{j}^{(n+1)}&=(C_{2(11)}^{mc})_{j}=
e(\tfrac{-1}{2}(\rho+a_n))(t_2-t_1)^{\rho+a_n}
\frac{\Gamma(\rho+a_n-\alpha_n^{(n)})}{\Gamma(-\alpha_n^{(n)})}
D_{j}^{(n)}\frac{\Gamma(\alpha_j^{(n)}-\rho)}{\Gamma(\alpha_j^{(n)}+a_n)}\\
&=e(\tfrac{-1}{2}(\alpha_{n}^{(n+1)}-\rho_{n+1}^{(n+1)}))
(t_2-t_1)^{\alpha_{n}^{(n+1)}-\rho_{n+1}^{(n+1)}}
\frac{\Gamma(-\alpha_{n+1}^{(n+1)})}{\Gamma(-\alpha_n^{(n+1)})}
D_{j}^{(n)}\frac{\Gamma(\alpha_j^{(n+1)}-\alpha_{n}^{(n+1)})}
{\Gamma(\alpha_j^{(n+1)}-\rho_{n+1}^{(n+1)})}
\end{split}
\end{equation}
Therefore, using these relations we can completely determine 
$C_{k}^{(n)}$ and $D_{1}^{(n)}$ in terms of the gamma function, 
if we know the initial data $C_{1}^{(2)}$ and $D_{1}^{(2)}$. 
The other connection coefficients 
$C_{i}^{(n)}$ and $D_{i}^{(n)}$ are obtained 
from  $C_{1}^{(n)}$ and $D_{1}^{(n)}$ 
by using symmetry of the Okubo system.
Let $\sigma^{\alpha}_{ij}$ be the operation which exchanges $\alpha_i^{(n)}$ and $\alpha_j^{(n)}$ $(i,j=1,\ldots,n-1)$.  
Then the canonical form of the Okubo system given by Theorem \ref{eq:cI} satisfies 
\begin{equation}
(x-T)\frac{d}{dx}\widetilde Y=A\widetilde Y
={\rm ad}(S_{ij})(\sigma^{\alpha}_{ij}A)\widetilde Y\qquad \widetilde Y={\rm ad}(S_{ij})(\sigma^{\alpha}_{ij}Y),
\end{equation}
where $S_{ij}$ is the permutation matrix corresponding the transposition $(ij)$ of matrix indices.
For the canonical solution matrix $Y$, the solution matrix ${\rm ad}(S_{ij})(\sigma^{\alpha}_{ij}Y)$
is also canonical solution matrix.
Therefore we obtain the other connection coefficients as 
\begin{equation}
C_i^{(n)}=\sigma_{1i}^{\alpha}C_1^{(n)},
\quad
D_i^{(n)}=\sigma_{1i}^{\alpha}D_1^{(n)},\quad (i=1,\ldots,n-1).
\end{equation}

The connection coefficients $C_{1}^{(2)}$ and $D_{1}^{(2)}$ 
for the Okubo system of type ${\rm (I)}_2$ are computed 
as follows.  
Firstly, 
the equation and the solution of the rank\,$1$ case are given by 
\begin{equation}
\frac{d}{dx}Y=\left(\frac{\alpha_1}{x-t_1}+\frac{\beta_1}{x-t_2}\right)Y
,\quad Y(x)=(x-t_1)^{\alpha_1}(x-t_2)^{\beta_1}.
\end{equation}
Then the canonical form of Okubo system ${\rm (I)}_2$ 
is given as a conjugation of \eqref{HGE}. 
The monodromy matrices ${\rm MC}_{\lambda}(e(\alpha_1),e(\beta_{1}))$ are 
given by 
\begin{equation}\label{HGEM}
M_1=
\begin{pmatrix}
e(\alpha_1) & e(\alpha_2-\rho_1)-1 \\
 0      & 1
\end{pmatrix},
\quad
M_2=
\begin{pmatrix}
 1 &  0 \\
e(\rho_1)(e(\alpha_1-\rho_1)-1) & e(\alpha_2)
\end{pmatrix}.
\end{equation}
Then the coefficients $r_1^{(2)}$ and $r_2^{(2)}$ are computed as 
\begin{equation}
\begin{split}
r_1^{(2)}=(t_1-t_2)^{\alpha_2-\rho_1}\widetilde B(\alpha_1-\rho_1,\rho_1+1), \quad
r_2^{(2)}=(\alpha_2-\rho_1)(t_2-t_1)^{\alpha_1-\rho_1}
\widetilde B(\alpha_2-\rho_1,\rho_1+1)
\end{split}
\end{equation}
From $r_1^{(2)}$and $r_2^{(2)}$, we obtain the connection coefficients
$C_{1}^{(2)}$ and $D_{1}^{(2)}$:
\begin{equation}
\begin{split}
C_{1}^{(2)}&=
\frac{(e(\alpha_2-\rho_1)-1)}{(e(\alpha_1)-1)}
\frac{r_1^{(2)}}{r_2^{(2)}} \\
&=\frac{1}{\alpha_2-\rho_1}
\frac{e(\alpha_1-\rho_1)-1}{e(\alpha_1)-1}
\frac{(t_1-t_2)^{\alpha_2-\rho_1}}{(t_2-t_1)^{\alpha_1-\rho_1}}
\frac{\Gamma(\alpha_1-\rho_1)\Gamma(\alpha_2+1)}
{\Gamma(\alpha_2-\rho_1)\Gamma(\alpha_1+1)} \\
&=e(\tfrac{-\rho_1}{2})
\frac{(t_1-t_2)^{\rho_2-\alpha_1}}{(t_2-t_1)^{\alpha_1-\rho_1}}
\frac{\Gamma(-\alpha_1)\Gamma(\alpha_2+1)}
{\Gamma(1+\rho_2-\alpha_1)\Gamma(1+\rho_1-\alpha_1)} \\
D_{1}^{(2)}&=
\frac{r_2^{(2)}}{r_1^{(2)}}\frac{e(\rho_1)(e(\alpha_2-\rho_1)-1)}{e(\alpha_2)-1}\\
&=e(\tfrac{\rho_1}{2})\frac{(t_2-t_1)^{\alpha_1-\rho_1}}{
(t_1-t_2)^{\alpha_2-\rho_1}}
\frac{\Gamma(-\alpha_2)\Gamma(\alpha_1+1)}
{\Gamma(\alpha_1-\rho_1)\Gamma(\alpha_1-\rho_2)}.
\end{split}
\end{equation}
Therefore we can obtain $C_{1}^{(n)}$ 
from $C_{1}^{(2)}$as follows: 
\begin{equation}
\small
\begin{split}
C_1^{(n)}=&(-1)^{n-2}\prod_{k=2}^{n-1}e(\tfrac{1}{2}(\alpha_{k}^{(k+1)}-\rho_{k+1}^{(k+1)}))(t_2-t_1)^{\rho_{k+1}^{(k+1)}-\alpha_{k}^{(k+1)}}
\frac{\Gamma(1+\alpha_{k}^{(k+1)}-\alpha_1^{(k+1)})}
{\Gamma(1+\rho_{k+1}^{(k+1)}-\alpha_1^{(k+1)})}
\frac{\Gamma(\alpha_{k+1}^{(k+1)}+1)}{\Gamma(\alpha_k^{(k)}+1)}C_1^{(2)}\\
=&(-1)^{n}e(\tfrac{1}{2}(\rho_2^{(n)}-\alpha_1^{(n)}-\alpha_n^{(n)}))
\frac{(t_1-t_2)^{\rho_2^{(n)}-\alpha_1^{(n)}}}{(t_2-t_1)^{\rho_2^{(n)}-\alpha_n^{(n)}}}
\Gamma(-\alpha_1^{(n)})\Gamma(\alpha_n^{(n)}+1)
\frac{\prod_{k=2}^{n-1}\Gamma(1+\alpha_{k}^{(n)}-\alpha_1^{(n)})}{\prod_{k=1}^n\Gamma(1+\rho_{k}^{(n)}-\alpha_1^{(n)})}
\end{split}
\end{equation}
The connection coefficient $D_{1}^{(n)}$ is obtained in the same way.  
Furthermore, exchanging $\alpha_1^{(n)}$, $\alpha_i^{(n)}$, 
we obtain $C_{i}^{(n)}$ and $D_{i}^{(n)}$.  
This completes the proof of Theorem \ref{mainI}.

\par\medskip
\noindent{\bf Case ${\rm I}^*$}:
We consider the Schlesinger system of rank $1$ and its solution defined by 
\begin{equation}
\frac{d}{dx}Y=\sum_{k=1}^n\frac{\alpha_k}{x-t_k}Y,\quad Y(x)=\prod_{k=1}^n(x-t_k)^{\alpha_k}.
\end{equation}
Then the Okubo system of type $({\rm I}^*)_n$ is constructed by $mc_{\mu}$, 
and its canonical solution matrix is expressed in the form 
$I^{\mu}(Y(x))R^{-1}$, where $R={\rm diag}(r_1,\ldots,r_n)$.
The components of the diagonal matrix $R$ are computed as
\begin{equation}
\begin{split}
r_i&=\lim_{x\to t_i}(x-t_i)^{-\alpha_i-\mu}\int_{L_i}(x-u)^{\mu}\prod_{k=1}^n(u-t_k)^{\alpha_k}
\frac{du}{u-t_i}\\
&=\lim_{x\to t_i}(x-t_i)^{-\alpha_i^{(n)}}\int_{L_i}(x-u)^{\rho_1^{(n)}}\prod_{k=1}^n
(u-t_k)^{\alpha_k^{(n)}-\rho_1^{(n)}}
\frac{du}{u-t_i}\\
&=\prod_{k\neq i,1\leq k\leq n}(t_i-t_k)^{\alpha_k^{(n)}-\rho_1^{(n)}}
\widetilde B(\alpha_i^{(n)}-\rho_1^{(n)},\rho_1^{(n)}+1).
\end{split}
\end{equation} 
Therefore the connection coefficients $C_{ij}$ are obtained as follows:
\begin{equation}
\begin{split}
C_{ij}&=\frac{e(\rho_1^{\delta(i>j)})(e(\alpha_j-\rho_1)-1)}{e(\alpha_i)-1}
\frac{r_i}{r_j}\\
&=\begin{cases}
&\ds e(\tfrac{\rho_1}{2})
\frac{\prod_{k\neq i,1\leq k\leq n}(t_i-t_k)^{\alpha_k-\rho_1}}
{\prod_{k\neq j,1\leq k\leq n}(t_j-t_k)^{\alpha_k-\rho_1}}
\frac{\Gamma(-\alpha_i)\Gamma(\alpha_j+1)}
{\Gamma(\alpha_j-\rho_1)\Gamma(1+\rho_1-\alpha_i)}
\quad (i<j)\\[10pt]
&\ds e(\tfrac{-\rho_1}{2})
\frac{\prod_{k\neq i,1\leq k\leq n}(t_i-t_k)^{\alpha_k-\rho_1}}
{\prod_{k\neq j,1\leq k\leq n}(t_j-t_k)^{\alpha_k-\rho_1}}
\frac{\Gamma(-\alpha_i)\Gamma(\alpha_j+1)}
{\Gamma(\alpha_j-\rho_1)\Gamma(1+\rho_1-\alpha_i)}
\quad(i>j)
\end{cases}
\end{split}
\end{equation}
This completes the proof of Theorem \ref{mainI*} .

\par\medskip
\noindent{\bf Cases ${\rm II}$ and ${\rm III}$}:
Keeping the notation of Section 2, 
we consider the Okubo system of type 
$({\rm II})_{2n}$ 
or $({\rm III})_{2n+1}$, and set 
$m=n$ or $m=n+1$ respectively. 
We denote by 
$\Psi^{(m+n)}(x)=(\psi_1(x),\ldots ,\psi_{m+n}(x))$ 
the canonical solution matrix 
of rank $m+n$. 
Taking account of the operation \eqref{mcadd}, 
the canonical solution matrix is constructed as follows

\begin{Proposition}\label{Th:Sol23}
$(1)$ Let $\Psi^{(2n-1)}(x)$ be the canonical solution matrix of type 
$({\rm III})_{2n-1}$.  
Then the solution matrix 
\begin{equation}\label{sol:3to2}
\Psi^{(2n)}(x)=GQ(x-t_2)^{\rho_2}I^{-b_{n-1}-\rho_2}
\big((x-t_2)^{b_{n-1}}\Psi^{(2n-1)}(x)\big)\cS\cX R^{-1}
\end{equation}
 is the canonical solution matrix of type 
$({\rm II})_{2n}$,where $R=(R_1,R_{(21)},R_{(22)})$.
\newline
$(2)$ Let $\Psi^{(2n)}(x)$ be the canonical solution matrix of type 
$({\rm II})_{2n}$.  
Then the solution matrix 
\begin{equation}\label{sol:2to3}
\Psi^{(2n+1)}(x)=GQ(x-t_1)^{\rho_2}
I^{-a_n-\rho_2}\big((x-t_1)^{a_n}\Psi^{(2n)}(x)\big)\cS\cX R^{-1}
\end{equation}
 is the canonical solution matrix of type $({\rm III})_{2n+1}$,
where $R=(R_{(11)},R_{(12)},R_2)$.
\end{Proposition}  

We determine the connection coefficients 
$C_{ij}=C_{ij}^{(n+m)}$ 
and $D_{ij}=D_{ij}^{(n+m)}$ for $\Psi^{(m+n)}(x)$. 
From Theorem \ref{connection}, we obtain the recurrence relations
\begin{equation}
\small
\begin{split}
C_{ij}^{(2n)}&=(C_{1(21)}^{mc})_{ij}=
e(\tfrac{1}{2}(\rho_2+b_n))(t_1-t_2)^{\rho_2+b_{n-1}}
\frac{\Gamma(\rho_2+b_{n-1}-\alpha_i)}{\Gamma(-\alpha_i)}
C_{ij}^{(2n)}\frac{\Gamma(\beta_j-\rho_2)}{\Gamma(\beta_j+c)}\\
&=e(\tfrac{1}{2}(\beta_{n}^{(2n)}-\rho_1^{(2n)}))
(t_1-t_2)^{\beta_{n}^{(2n)}-\rho_1^{(2n)}}
\frac{\Gamma(-\alpha_i^{(2n-1)})}{\Gamma(-\alpha_i^{(2n)})}
C_{ij}^{(2n-1)}\frac{\Gamma(\beta_j^{(2n)}-\beta_{n}^{(2n)})}
{\Gamma(\beta_j^{(2n)}-\rho_1^{(2n)})},\\
D_{ij}^{(2n)}&=(C_{(21)1}^{mc})_{ij}=
-e(\tfrac{-1}{2}(\rho_2+b_{n-1}))(t_1-t_2)^{-\rho_2-b_{n-1}}
\frac{\Gamma(1+\rho_2-\beta_i)}{\Gamma(1-\beta_i-b_{n-1})}D_{ij}^{(2n)}
\frac{\Gamma(\alpha_j-\rho-b_{n-1}+1)}{\Gamma(\alpha_j+1)}\\
&=-e(\tfrac{-1}{2}(\beta_{n}^{(2n)}-\rho_1^{(2n)}))
(t_1-t_2)^{\rho_1^{(2n)}-\beta_{n}^{(2n)}}
\frac{\Gamma(1+\beta_{n}^{(2n)}-\beta_i^{(2n)})}
{\Gamma(1+\rho_1^{(2n)}-\beta_i^{(2n)})}
D_{ij}^{(2n-1)}\frac{\Gamma(\alpha_j^{(2n)}+1)}{\Gamma(\alpha_j^{(2n-1)}+1)},
\end{split}
\end{equation}
\begin{equation}
\small
\begin{split}
C_{ij}^{(2n+1)}&=
-e(\tfrac{1}{2}(\alpha_{n+1}^{(2n+1)}-\rho_1^{(2n+1)}))
(t_2-t_1)^{-\alpha_{n+1}^{(2n+1)}+\rho_1^{(2n+1)}}
\frac{\Gamma(1+\alpha_{n+1}^{(2n+1)}-\alpha_i^{(2n+1)})}
{\Gamma(1+\rho_1^{(2n+1)}-\alpha_i^{(2n+1)})}C_{ij}^{(2n)}
\frac{\Gamma(\beta_j^{(2n+1)}+1)}{\Gamma(\beta_j^{(2n)}+1)},\\
D_{ij}^{(2n+1)}
&=e(\tfrac{-1}{2}(\alpha_{n+1}^{(2n+1)}-\rho_1^{(2n+1)}))
(t_2-t_1)^{\alpha_{n+1}^{(2n+1)}-\rho_1^{(2n+1)}}
\frac{\Gamma(-\beta_i^{(2n+1)})}{\Gamma(-\beta_i^{(2n)})}
D_{ij}^{(2n)}\frac{\Gamma(\alpha_j^{(2n+1)}-\alpha_{n+1}^{(2n+1)})}
{\Gamma(\alpha_j^{(2n+1)}-\rho_1^{(2n+1)})}. 
\end{split}
\end{equation}

Therefore using these relations, we can completely determine 
$C_{11}^{(m+n)}$ and $D_{11}^{(m+n)}$ in terms of the gamma function, 
if we know the initial data $C_{11}^{(2)}$ and $D_{11}^{(2)}$. 
The other connection coefficients 
$C_{ij}^{(m+n)}$ and $D_{ji}^{(m+n)}$ are obtained 
from  $C_{11}^{(m+n)}$ and $D_{11}^{(m+n)}$ 
by exchanging 
$\alpha_1$, $\beta_1$ and $\alpha_i$, $\beta_j$. 
The connection coefficients $C_{11}^{(2)}$ and $D_{11}^{(2)}$ 
for the Okubo system \eqref{HGE} of type ${\rm (II)}_2$ are given 
as follows. 
We remark that we use characteristic exponent $\beta_1^{(2)}$ instead of $\alpha_2^{(2)}$.

\begin{equation}
\begin{split}
C_{11}^{(2)}
&=-e(\tfrac{-\rho_1}{2})
\frac{(t_1-t_2)^{\beta_1-\rho_1}}{(t_2-t_1)^{\alpha_1-\rho_1}}
\frac{\Gamma(-\alpha_1)\Gamma(\beta_1+1)}
{\Gamma(\beta_1-\rho_1)\Gamma(1-\alpha_1+\rho_1)} \\
D_{11}^{(2)}
&=-e(\tfrac{\rho_1}{2})\frac{(t_2-t_1)^{\alpha_1-\rho_1}}{(t_1-t_2)^{\beta_1-\rho_1}}
\frac{\Gamma(-\beta_1)\Gamma(\alpha_1+1)}
{\Gamma(\alpha_1-\rho_1)\Gamma(1-\beta_1-\rho_1)}.
\end{split}
\end{equation}
Therefore we can obtain $C_{11}^{(m+n)}$ 
from $C_{11}^{(2)}$as follows: 
\begin{equation}
\begin{split}
C_{11}^{(2n)}&=(-1)^{n-1}
\prod_{k=2}^ne(\tfrac{1}{2}(\alpha_k^{(2k-1)}-\rho_1^{(2k-1)}+\beta_{k}^{(2k)}-\rho_1^{(2k)}))
\frac{(t_1-t_2)^{\beta_k^{(2k)}-\rho_1^{(2k)-1}}}{(t_2-t_1)^{\alpha_k^{(2k-1)}-\rho_1^{(2k-1)}}}
\\
&\prod_{k=2}^{n}
\frac{\Gamma(-\alpha_1^{(2k)})}{\Gamma(-\alpha_1^{(2k-1)})}
\frac{\Gamma(\beta_1^{(2k)}-\beta_{k}^{(2k)})}
{\Gamma(\beta_1^{(2k)}-\rho_1^{(2k)})}
\frac{\Gamma(1+\alpha_{k}^{(2k-1)}-\alpha_1^{(2k-1)})}
{\Gamma(1+\rho_1^{(2k-1)}-\alpha_1^{(2k-1)})}
\frac{\Gamma(\beta_1^{(2k-1)}+1)}{\Gamma(\beta_1^{(2k-2)}+1)}
C_{11}^{(2)}\\
&=(-1)^{n-1}e(\tfrac{\alpha_2^{(3)}+\beta_2^{(4)}-\rho_1^{(2n)}-\rho_2^{(2n)}}{2})
\frac{(t_1-t_2)^{\beta_2^{(4)}-\rho_2^{(2n)}}}{(t_2-t_1)^{\alpha_2^{(3)}-\rho_1^{(2n)}}}
\frac{\Gamma(\beta_1^{(2n)}+1)\Gamma(-\alpha_1^{(2n)})}
{\Gamma(\beta_{1}^{(4)}+1)\Gamma(-\alpha_1^{(3)})}\\
&\prod_{k=2}^{n}
\frac{\Gamma(\beta_1^{(2k)}-\beta_{k}^{(2k)})}
{\Gamma(\beta_1^{(2k)}-\rho_1^{(2k)})}
\frac{\Gamma(1+\alpha_{k}^{(2k-1)}-\alpha_1^{(2k-1)})}
{\Gamma(1+\rho_1^{(2k-1)}-\alpha_1^{(2k-1)})}
C_{11}^{(2)}\\
&=(-1)^{n}
e(\tfrac{1}{2}(\rho_3^{(2n)}-\alpha_1^{(2n)}-\beta_1^{(2n)}))
\frac{(t_1-t_2)^{\rho_3^{(2n)}-\alpha_1^{(2n)}}}{(t_2-t_1)^{\rho_3^{(2n)}-\beta_1^{(2n)}}}
\frac{\Gamma(\beta_1^{(2n)}+1)\Gamma(-\alpha_1^{(2n)})}
{\Gamma(1+\rho_1^{(2n)}-\alpha_1^{(2n)})\Gamma(\beta_1^{(2n)}-\rho_1^{(2n)})}\\
&\prod_{k=2}^{n}
\frac{\Gamma(1+\alpha_{k}^{(2n)}-\alpha_1^{(2n)})}
{\Gamma(1+\rho_1^{(2n)}+\rho_2^{(2n)}-\alpha_1^{(2n)}-\beta_{k}^{(2n)})}
\frac{\Gamma(\beta_1^{(2n)}-\beta_{k}^{(2n)})}
{\Gamma(\beta_1^{(2n)}+\alpha_{k}^{(2n)}-\rho_1^{(2n)}-\rho_2^{(2n)})}\\
\end{split}
\end{equation}

The connection coefficients $D_{11}^{(2n)}$,\,$C_{11}^{(2n+1)}$, 
and $D_{11}^{(2n+1)}$ are obtained in the same way.  
Furthermore, 
exchanging $\alpha_1^{(m+n)}$, $\beta_1^{(m+n)}$ and 
$\alpha_i^{(m+n)}$, $\beta_j^{(m+n)}$ in 
$C_{11}^{(m+n)}$, $D_{11}^{(m+n)}$, 
we obtain 
$C_{ij}^{(m+n)}$ and $D_{ji}^{(m+n)}$.  
This completes the proof of Theorem \ref{mainII,III}. 

\section*{Acknowledgements}
The author is grateful to Professors Toshio Oshima, 
Masatoshi Noumi 
and Yoshishige Haraoka for their valuable comments 
and suggestions.

\end{document}